\documentclass{amsart}
\usepackage{color,graphicx,enumerate,amssymb}
\usepackage{hyperref}
\usepackage{graphicx}
\usepackage{eucal}
\usepackage[utf8]{inputenc}
\usepackage{ulem}

\newcommand{\bR}{{\mathbb R}}
\newcommand{\bZ}{{\mathbb Z}}
\newcommand{\bN}{{\mathbb N}}

\newcommand{\bH}{{\mathbb H}}

\newcommand{\cB}{{\mathcal B}}
\newcommand{\cG}{{\mathcal G}}
\newcommand{\cP}{{\mathcal P}}
\newcommand{\cQ}{{\mathcal Q}}
\newcommand{\cR}{{\mathcal R}}
\newcommand{\cF}{{\mathcal F}}
\newcommand{\cH}{{\mathcal H}}
\newcommand{\cL}{{\mathcal L}}

\newcommand{\uu}{\mathbf{ u}}
\newcommand{\vv}{\mathbf{ v}}
\newcommand{\ww}{\mathbf{ w}}

\newcommand{\FF}{\mathbf{F}}
\newcommand{\GG}{\mathbf{G}}

\newcommand{\ee}{\mathbf{ e}}

\newcommand{\e}{\varepsilon}
\newcommand{\kp}{\kappa}

\newcommand{\supp}{\operatorname{supp}}
\newcommand{\dist}{{\rm dist}}
\newcommand{\rdiv}{{\rm div\,}}

\newcommand{\ra}{\rightarrow}

\newcommand\norm[1]{\Arrowvert {#1}\Arrowvert}

\newtheorem{theorem}{Theorem}[section]

\newtheorem{definition}[theorem]{Definition}

\newtheorem{lemma}[theorem]{Lemma}

\newtheorem{proposition}[theorem]{Proposition}
\newtheorem{remark}[theorem]{Remark}

\def\XXint#1#2#3{{\setbox0=\hbox{$#1{#2#3}{\int}$ }
\vcenter{\hbox{$#2#3$ }}\kern-.6\wd0}}

\def\permut#1#2{\left(\begin{array}{c}#1\\#2\end{array}\right)}

\title [Higher Regularity of the Free Boundary]{Higher Regularity of the Free Boundary in a Semilinear System}

\author{Morteza Fotouhi  and Herbert Koch }
\address{Department of Mathematical Sciences, Sharif University of Technology, Tehran, Iran}
\email{fotouhi@sharif.edu}
\address{Mathematisches Institut, Universit\"at Bonn, Endenicher Allee 60, 53115 Bonn, Germany}
\email{koch@math.uni-bonn.de}

\date{\today}

\begin{document}
    
\begin{abstract} 
In this paper we are concerned with  higher regularity properties  of the  elliptic system 
\[
\Delta\uu= |\uu|^{q-1}\uu\chi_{\{|\uu|>0\}},\qquad\uu=(u^1,\dots,u^m)
\]
for $0\leq q<1$. We show  analyticity of the regular part of the free boundary $\partial\{|\uu|>0\}$, analyticity of $|\uu|^{\frac{1-q}2} $ and $ \frac{\uu}{|\uu|}$  up to the regular part of the free boundary. 
Applying  a variant of the partial hodograph-Legendre transformation  and the implicit function theorem, we arrive at a degenerate equation, which introduces substantial  challenges  to be dealt with.

Along the lines of our  study, we also establish a Cauchy-Kowalevski type statement to show the local existence of solution when the free boundary and the restriction of $ \frac{\uu}{|\uu|} $ from both sides to the free boundary are given as analytic data. 
\end{abstract}
%
%\clearpage\maketitle
%\thispagestyle{empty}
%
%\tableofcontents
%

\maketitle

%%%%%%%%%%%%%%%%%%%%%%%
%%%%%%%%SECTION%%%%%%%%%%
%%%%%%%%%%%%%%%%%%%%%%%
\section{Introduction}\label{introduction}

%%%%%%%%SUBSECTION%%%%%%%%
%%%%%%%%%%%%%%%%%%%%%%%
\subsection{Problem setting}
In this paper we  study the higher regularity properties of the free boundary of solutions to the following elliptic system 
\begin{equation}\label{elliptic-system}
\Delta\uu= |\uu|^{q-1}\uu\chi_{\{|\uu|>0\}},\qquad\uu=(u^1,\dots,u^m) , 
\end{equation}
where $0\leq q<1$,  $\uu: B_1\subset\bR^n\longrightarrow\bR^m$, $n\geq2$, $m\geq1$, 
and  $|\cdot|$ stands for the Euclidian norm on the respective spaces.

Problem \eqref{elliptic-system}  is studied in  \cite{asuw15, FSW20} and  in the scalar case ($m=1$)  in \cite{FS17}, 
where an  optimal growth rate for the solutions is proven as well as $C^{1,\beta}$-regularity of the free boundary at  {\it regular} points.
In addition, the scalar case with a further assumption of a positive solution has been extensively studied in \cite{alt1986free, phillips1983hausdoff, phillips1983minimization}; for the case of  fully nonlinear operators see  \cite{wu2022fully}, and for the parabolic version see \cite{aleksanyan2022regularity}.

Our aim with  this paper is to improve the  existing  results by showing that the regular part of the  free boundary  is analytic.

Throughout the paper, we fix the following constants:
\begin{align*}%\label{rel-alpha-kappa}
&\kp=\frac2{1-q},
&\alpha=(\kp(\kp-1))^{-\kp/2}.
\end{align*}
We observe that $ \kappa \in [2,\infty)$ for $ q \in [0,1)$. %, and $ \kappa \in (1,2) $ if $ -1<q <0$.  
We denote by $\Gamma(\uu)=\partial\{|\uu|>0\}$ the free boundary and define the subset $\Gamma^\kp(\uu)$ as follows
\[
\Gamma^\kp(\uu):=\left\{z\in\Gamma(\uu): \partial^\mu\uu(z)=0,\text{ for all }|\mu|<\kp\right\}.
\]
In \cite{asuw15, FSW20}, it is  proven that the solution of \eqref{elliptic-system} has $\kp$-growth close to points in $\Gamma^\kp(\uu)$, indeed $\norm{\uu}_{L^\infty(B_r(z))}=O(r^\kp)$ when $z\in \Gamma^\kp(\uu)$. 
With this regularity, we can get a sense of the blow-up limit  of $\uu(rx+z)/r^\kp$  at each point $z\in\Gamma^\kp$, which is a $\kp$-homogeneous global solutions of \eqref{elliptic-system} (thanks to the energy monotonicity).
When the domain is a plane, i.e. $n=2$, all $\kp$-homogeneous global solutions are classified (see \cite{FS17, FSW20}).
An important class of global solutions  is the class of  
 {\it half-space} solutions, defined  by
\[
\bH:=\left\{
x\mapsto\alpha\max(x\cdot\nu,0)^\kp\ee: \nu\in\bR^n 
\hbox{ and } \ee\in\bR^m \hbox{ are unit vectors}
 \right\}.
\]
\begin{definition}\label{def-regular-sing-FB}
We denote by $\cR_{\uu}$ the set of all (regular free boundary) points $x_0\in\Gamma(\uu)$ such that  at least one blow-up limit of $\uu$ at $x_0$ is in $\bH$. 
\end{definition}
 The terminology {\it regular free boundary points} is justified by the  uniqueness of blow-up, as it has been proved in \cite{asuw15, FSW20}. Again by  \cite{asuw15, FSW20}  the set of regular free boundary points $\cR_{\uu}$ is open relative to $\Gamma(\uu)$ and locally is a $C^{1,\beta}$-manifold. Our first main result in this paper is presented in the following theorem.

\begin{theorem}\label{main-result}
The set of regular free boundary points $\cR_{\uu}$ is locally an analytic manifold.
In addition,   $\uu/|\uu|$ and $ |\uu|^{\frac1{\kappa}}  $ are analytic in $\overline{\{|\uu|>0\}}$ in a neighborhood of $\cR_{\uu}$.  In particular, when $\kp$ is integer the solution will be analytic.
\end{theorem}

Our proof is based on changing coordinates  with a  partial  Hodograph-Legendre transform involving dependent and independent variables so that we obtain a degenerate elliptic problem on a fixed domain. 
This method first has been  introduced in \cite{kinderlehrer1977regularity} to show the analyticity of the  free boundary 
in the classical obstacle problem.  
It has been also applied   to the  thin obstacle problem  or the Signorini problem  in \cite{koch2015higher, koch2017variable, koch2019higher} where  it leads to a  fully nonlinear Baoundi-Grushin equation. 
 Our transform is closer to Section 2.1, Theorem 1.4 of Friedman \cite{MR1009785}.

In our case the approach leads to a nonlinear degenerate  PDE in a half space 
similar  to the stationary  case of  the porous medium equation studied  by Daskaloupolos and Hamilton \cite{MR1474840} 
and  with different approach by the second author \cite{kochnon}. 
We need regularity results for such linear and nonlinear problems. 
According to \cite{MR1474840}, estimates in H\"older spaces can be found for the parabolic problem of the porous medium equation.
These estimates in Sobolev spaces have been proven by the second author \cite{kochnon}.
To keep the paper self-contained  we will provide a complete  proof for the elliptic case in Sobolev spaces.

In this study, we deduce analyticity of the free boundary by an application of  the implicit function theorem instead of direct and tedious elliptic estimates.
In doing so we follow the ideas that have been used in \cite{angenent1990nonlinear, angenent1990parabolic, koch2012geometric, koch2017variable} and consider the transformed fully nonlinear equation as a perturbation of a linear degenerate operator between appropriate function spaces (Section \ref{section:main-result}).

We also establish a  Cauchy-Kowalevski type statement to show the existence of  a solution  for \eqref{elliptic-system}  when the free boundary is known and assumed to be analytic.

\begin{theorem}\label{thm:existence}
Let  $\Gamma\subset \bR^n$ be any given  oriented analytic hypersurface, and 
$ V_\pm  : \Gamma \to \mathbb{S}^{m-1}$ given  analytic maps. Then on each side of $\Gamma $ 
there is  a unique local  solution $\uu_\pm $,  of \eqref{elliptic-system}, that   are   analytic up to the boundary and satisfy
\[ \frac{\uu_\pm(x)}{| \uu_\pm(x)|}\to V_\pm   \]
as $ x \to \Gamma $ from respective  sides. 
\end{theorem}

%%%%%%%%SUBSECTION%%%%%%%%
%%%%%%%%%%%%%%%%%%%%%%%
\subsection{Notation} 
For clarity of exposition we shall introduce some notation and definitions here  that are used frequently in the paper.

Throughout this paper, $\bR^n$ will be equipped with the Euclidean inner product $x \cdot y$ and the induced norm $|x|$, $B_r(x_0)$ will denote the open $n$-dimensional ball of center $x_0$, radius $r$ with the boundary $\partial B_r(x_0)$. 
In addition, $B_r=B_r(0)$, $\partial B_r=\partial B_r(0)$ and $B_r^+=B_r\cap\{x_n>0\}$.
%For a set $A$, $d(x,A)$  stands for the   distance between  $x$ and  $A$.
 We shall also denote  the $n$-dimensional Hausdorff measure  by  $\cH^n$.

%Also, we will denote  the derivative matrix of $\uu$ by $\nabla\uu=[\partial_i u_j]_{1\leq i\leq n,1\leq j\leq m}$ with the notation
% \begin{align*}
%|\nabla\uu|^2&=\sum_{i=1}^m|\nabla u_i|^2,&\nabla\uu:\nabla\vv=\sum_{i=1}^m(\nabla u_i\cdot\nabla v_i),\\
%\nabla\uu\cdot\xi&=\xi^t\nabla\uu=(\nabla u_1\cdot\xi,\cdots,\nabla u_m\cdot\xi),& \text{ for all }\xi\in\bR^n.
%\end{align*}

Here we introduce the multi-indices notation that is used in Section  \ref{section:main-result}.
For a multi-index $\mu=(\mu_1,\cdots,\mu_n)\in \bZ^n_+$, we denote  
\[
\mu!=\prod_{i=1}^n\mu_i!, \qquad |\mu|=\mu_1+\dots+\mu_n, \qquad x^\mu=x_1^{\mu_1}\cdots x_n^{\mu_n}.
\]
Let $\mu$ and  $\sigma$ be two multi-index, we say $\sigma\leq \mu$ when $\sigma_i\leq\mu_i$ for $i=1,\cdots,n$. 
We also use notation $0<\sigma$ to say $0<|\sigma|$ as well as $\sigma<\mu$ stands for the case $\sigma\leq \mu$ and $\sigma\neq\mu$.
When $\sigma\leq \mu$, we will denote
\[
 \permut\mu\sigma=\frac{\mu!}{\sigma!(\mu-\sigma)!}=\prod_{i=1}^n\permut{\mu_i}{\sigma_i}.
\]
The following relations is also helpful in our proofs
\begin{equation*}\label{rel-permut}
\sum_{|\sigma|=s}\permut\mu\sigma = \permut{|\mu|}s,
\end{equation*}
which has the consequence
\begin{equation*}%\label{rel-permut}
\permut\mu\sigma \le \permut{|\mu|}{|\sigma|}.
\end{equation*}

%%%%%%%%SUBSECTION%%%%%%%%
%%%%%%%%%%%%%%%%%%%%%%%
\subsection{Outline of the paper}
The paper is organized as follows: 
In Section \ref{section:regularity}, we obtain some results concerning  the regularity of  solutions near the regular points of  the free boundary, that are essential  for the rest  of the paper.
In Section \ref{Section:Hodograph} we invoke  the  classical partial hodograph transformation and show that it is a diffeomorphism. 
 Section \ref{section:degenerate} is devoted to  improve the a priori regularity of  the Legendre functions.  
Section \ref{section:main-result} contains    the proof of our main results, Theorem \ref{main-result} and Theorem \ref{thm:existence}. 
Finally, in Appendix we give a short proof for the regularity of the linear degenerate equation that we obtain after the transformation.

%%%%%%%%%%%%%%%%%%%%%%%
%%%%%%%%SECTION%%%%%%%%%%
%%%%%%%%%%%%%%%%%%%%%%%

\section{Some regularity results for solutions }\label{section:regularity}

In this section, we give a brief overview of some tools and known results about the regularity of 
solutions and the  free boundary for problem \eqref{elliptic-system}. 
We shall also introduce some new results, 
which are crucial in next sections to follow. 
Indeed,   near the regular parts of the  free boundary, solutions behave more smoothly and we can find out some more  information. 
We start  with introducing the monotonicity formula, that are the corner stone in the regularity theory of free boundary. 

\begin{proposition}[Proposition 2.3 in \cite{FSW20}]\label{monotonicity formula}
Let $\uu$ be a solution of \eqref{elliptic-system} in $B_{r_0}(x_0)$ and let 
$$W(\uu,x_0,r)=\frac1{r^{n+2\kp-2}}\int_{B_r(x_0)}\big(|\nabla\uu|^2+\frac2{1+q}|\uu|^{1+q}\big)\,dx-\frac\kp{r^{n+2\kp-1}}\int_{\partial B_r(x_0)}|\uu|^2\,d\cH^{n-1}.$$
(i) For $0<r<r_0$, the energy function $W(\uu,x_0,r)$ is non-decreasing.\\
(ii) If $W(\uu,x_0,r)$ is constant for $r_1<r<r_2$, then $\uu$ is $\kp$-homogeneous with respect to $x_0$, i.e.
\[\uu(x_0+rx)=(\frac{r}{r_1})^\kp\uu(x_0+r_1x),\quad x\in B_{1}(0),\text{ and }r_1<r<r_2.\]
(iii) The function $x\mapsto W(\uu,x,0+)$ is upper-semicontinuous.
\end{proposition}

\medskip

The next proposition presents the optimal growth for the solution near $\Gamma^\kp(\uu)$.

\begin{proposition}[Theorem 1.2 in  \cite{FSW20}]\label{growth-estimate}
Let $\uu$ be a solution of \eqref{elliptic-system}  and $x_0\in \Gamma^\kp(\uu)$. Then there exists the positive constant $C$ such that
\[
\sup_{B_r(x_0)}|\uu|\leq Cr^\kp.
\]
Moreover, the  constant  $C$ is uniform in any compact subset of $\Gamma^\kp(\uu)$.
\end{proposition}

As of now, the result for the regularity of  the free boundary  is as follows.

\begin{proposition}[Theorem 1.4 in \cite{FSW20}]\label{Regularity of free boundary}
The set of regular free boundary points $\cR_{\uu}$ is an open set in $\Gamma(\uu)$ and is locally  a $C^{1,\beta}$-manifold for some $\beta\in (0,1)$ depending only on $n$, $m$ and $q$.
\end{proposition}

In what follows,  we shall study  the behaviour of solutions near the regular parts of  the  free boundary. 
From now on, we assume without loss of generality that $0\in\cR_{\uu}$ and  that the blowup limit at $x=0$ is 
\begin{equation}\label{blowup-assumption}
\uu_r(x)=\frac{\uu(rx)}{r^\kp}\ra \alpha(x\cdot\ee_n)_+^\kp\ee_1.
\end{equation}
We also denote the  inner  unit  normal vector on the free boundary by $\nu$. 
According to $C^{1,\beta}$-regularity for the free boundary, Proposition \ref{Regularity of free boundary}, we know that $\cR_{\uu} \ni    x_0\mapsto \nu(x_0)$ is $C^{0,\beta}$ around the origin on the free boundary. 
Moreover, the blowup at a point $x_0\in\cR_\uu$ must be $\alpha(x\cdot \nu(x_0))_+^\kp\ee(x_0)$ for some unit vector $\ee(x_0)\in\bR^m$, in particular $\nu(0)=\ee_n$ and $\ee(0)=\ee_1$.

As a first step, we prove an auxiliary lemma that appears several times later in the arguments.
\begin{lemma}\label{auxiliary lemma}
Let $\uu$ be a solution of \eqref{elliptic-system}  in $B_2(x_0)$ with $0\leq q<1$ and $x_0\in \cR_\uu$ is a regular point of free boundary.
Consider an arbitrary sequence 
$$\{|\uu|>0\}\cap B_1\ni x_i\ra x_0$$ 
and let $d_i:=\dist(x_i,\Gamma)$.
Then  we find out 
\[
\uu_i(x):=\uu(x_i+d_ix)/d_i^\kp\ra\alpha\left((x+\nu(x_0))\cdot\nu(x_0)\right)_+^\kp\ee(x_0),
\]
in $C^{2,\gamma}(B_{1/2})$ for any $0<\gamma<1$.
\end{lemma}
\begin{proof}
For convenience we suppose  $x_0=0$ and that assumption \eqref{blowup-assumption} holds.
Choose $y_i\in \partial B_{d_i}(x_i)\cap \Gamma$. 
The regularity of  the free boundary   implies that 
\begin{equation}\label{auxl-lem-eq:4}
\nu(y_i)=(x_i-y_i)/d_i.
\end{equation} 
Since the map to the inner unit normal vector $ \nu$ is continuous  we have (note that $x_i \to 0$ also $ y_i \to 0$)
\begin{equation*}%\label{auxl-lem-eq:4}
(x_i-y_i)/d_i=\nu(y_i)\ra\nu(0)=\ee_n.
\end{equation*}
According to the growth estimate in Proposition \ref{growth-estimate}   at the point $y_i$, and  that $\cR_\uu$ is open in $\Gamma(\uu)$, we can   apply the growth estimate at $\Gamma^\kp(\uu)$ for large values of $i$,  to obtain the estimate 
\begin{equation*}%\label{eq:uni}
|\uu(x_i+d_ix)|\leq C|x_i+d_ix-y_i|^\kp \leq Cd_i^\kp|x+\nu(y_i)|^\kp.
\end{equation*}
Thus $\uu_i$ is a bounded sequence of solutions of $\Delta \uu_i=|\uu_i|^{q-1}\uu_i$ in $B_1\subset\{|\uu_i|>0\}$. 
The classical  regularity theory yields that $\{\uu_i\}$ is bounded in $C^{1,\gamma}(B_{1})$ and $C^{2,\gamma}(B_{1/2})$ for any $0<\gamma<1$.
Passing to a subsequence, we may assume that 
\[
\uu_i(x)\ra \uu_0(x),
\]
where the convergence is in $C^{1,\gamma}(B_1)$ and $C^{2,\gamma}(B_{1/2})$.

Now for $\rho$ positive and small enough,  $B_\rho(0)\cap\Gamma(\uu)$  consists only of regular points ($\cR_\uu$ is open in $\Gamma(\uu)$), and hence 
\begin{equation}\label{aux-lem-eq:1}
 \lim_{r\ra0^+}W(\uu,x,r)=\omega_q, \text{ for }x\in B_\rho(0)\cap\Gamma(\uu),
\end{equation}
 where $\omega_q$ is the energy of half-space solutions which is a universal constant depending only on $n$ and $q$. 
 We know that for every $x$, the function $r\mapsto W(\uu,x,r)$ is nondecreasing in $r$.
 Hence  by Dini's monotone convergence theorem, the convergence in \eqref{aux-lem-eq:1}   will be uniform in $ x$, for smaller $\rho$.     Thus 
 \[
 W(\uu_0,-\ee_n,r)=\lim_{i\ra\infty}W(\uu_i, -\nu(y_i),r)=\lim_{i\ra\infty}W(\uu, y_i, rd_i)=\omega_q,
 \]
for any $r>0$. 
Hence $\uu_0$ is a homogeneous function with respect to point $-\ee_n$.

On the other hand, $y_i$ is a regular point of free boundary for $\uu$.  
Therefore, for every $\varepsilon>0$, there exist  $\delta>0$ such that
$$\{(x-y_i)\cdot\nu(y_i)<-\varepsilon|x-y_i|\}\cap B_\delta(0)\subset\{\uu=0\},$$
when $y_i$ is close enough to $x_0=0$. 
Invoking the relation \eqref{auxl-lem-eq:4} and definition of $\uu_i$ to get 
\[
\{(x+\nu(y_i))\cdot\nu(y_i)<-\varepsilon|x+\nu(y_i)|\}\cap B_{\delta/d_i}(0)\subset\{\uu_i=0\},
\]
and passing to the limit,
\[
\{(x+\ee_n)\cdot \ee_n<-\varepsilon|x+\ee_n|\}\subset\{\uu_0=0\}.
\]
Since $\varepsilon$ was arbitrary, we conclude that 
\begin{equation}\label{aux-lem-eq:2}
\{(x+\ee_n)\cdot \ee_n<0\}\subset\{\uu_0=0\}.
\end{equation}
Now according to Proposition 4.6 in \cite{FSW20}, we infer that $\uu_0$ must be a half-space solution with respect to the point $-\ee_n$
and \eqref{aux-lem-eq:2} yields that 
\begin{equation*}%\label{eq:uu0}
\uu_0(x)=\alpha ((x+\ee_n)\cdot\ee_n)_+^\kp\ee,
\end{equation*}
for some unit vector $\ee\in\bR^m$. 
(Recall that the nondegeneracy property, Proposition 4.1 in  \cite{FSW20},  necessitates that $\uu_0\not\equiv0$.) 
We must show that $\ee=\ee_1$. 
To do this note that
\begin{equation}\label{aux-lem-eq:3}
\alpha\ee=\uu_0(0)=\lim_{i\ra\infty}\uu_i(0)=\lim_{i\ra\infty}\uu(x_i)/d_i^\kp=\lim_{i\ra\infty}\frac{\uu(x_i)}{|x_i|^\kp}\cdot\frac{|x_i|^\kp}{d_i^\kp},
\end{equation}
and if $\frac{x_i}{|x_i|}\ra x_*$ in a subsequence, by assumption \eqref{blowup-assumption} we get
\[
\lim_{i\ra\infty}\frac{\uu(x_i)}{|x_i|^\kp}=\alpha (x_*\cdot\ee_n)_+^\kp\ee_1.
\]
This together with \eqref{aux-lem-eq:3} implies that $\ee=\ee_1$.
\end{proof}

Now,  in the following proposition we find out more information about the behavior of solution near regular free boundary points.

 \begin{proposition}\label{property-near-regular}
Let $\uu$ be a solution of \eqref{elliptic-system}  and suppose  \eqref{blowup-assumption} holds. 
Then the following properties hold  in a neighborhood of $x=0$:
\begin{enumerate}[(i)]
\item
$u^1\geq0$.

\item
$\frac{u^j}{|\uu|}\in C^0(\overline{\{|\uu|>0\}})$ and on the free boundary point $x_0\in\Gamma(\uu)$, we have 
\[\lim_{\{|\uu|>0\}\ni x\ra x_0}\frac{u^j(x)}{|\uu(x)|}=\ee(x_0)\cdot\ee_j.\]

\item
$\{u^1>0\}=\{|\uu|>0\}$.

%\item
%$\nabla (u^1)^{1/\kp}\in C^0(\overline{\{|\uu|>0\}})$ and on the free boundary point $x_0\in\Gamma(\uu)$, we have 
%\[\lim_{\{|\uu|>0\}\ni x\ra x_0}\nabla (u^1)^{1/\kp}(x)=\alpha^{1/\kp}\nu(x_0)(\ee(x_0)\cdot\ee_1)^{1/\kp}.\]

\item
For all  $x_0\in\Gamma(\uu)$, we have 
\[
\lim_{\{|\uu|>0\}\ni x\ra x_0}|\uu(x)|^{-\frac{1+q}2}\partial_{i}\uu(x)=\kp\alpha^{1/\kp}(\nu(x_0)\cdot\ee_i)\ee(x_0).
\]
In particular, $\nabla (u^1)^{1/\kp}\in C^0(\overline{\{|\uu|>0\}})$.

\medskip
\item 
For all  $x_0\in\Gamma(\uu)$ and $1\leq i,j\leq n$, we have 
\[
\lim_{\{|\uu|>0\}\ni x\ra x_0}|\uu(x)|^{-q}\partial_{ij}\uu(x)=(\nu(x_0)\cdot\ee_i)(\nu(x_0)\cdot\ee_j)\ee(x_0).
\]

\end{enumerate}
\end{proposition}
\begin{proof}
The proof is based on the application of Lemma \ref{auxiliary lemma}. In all parts, we consider a neighborhood in which the free boundary consists only of regular points.

{\it (i)} Assume that  there exists a  sequence $x_i\ra0$ such that $u^1(x_i)<0$. 
For  $\uu_i(x):=\uu(x_i+d_ix)/d_i^\kp$  we obtain, as in Lemma \ref{auxiliary lemma}, 
\[
\uu_i(x)\ra\uu_0(x)=\alpha((x+\ee_n)\cdot\ee_n)_+^\kp\ee_1.
\]
Thus
\[
\alpha=\uu_0(0)\cdot\ee_1=\lim_{i\ra\infty}\uu_i(0)\cdot\ee_1=\lim_{i\ra\infty}\uu(x_i)\cdot\ee_1/d_i^\kp\leq 0,
\]
which is a contradiction.

%%%%%%%%%%%%%%%

{\it (ii)}
It is obvious that $\frac{u^j}{|\uu|}$ is smooth  inside the  set  $\{|\uu|>0\}$. We show that it can extend continuously on the free boundary. 
If   $x_i$ is a sequence converging to $x_0$, a point on the  free boundary.  Define the corresponding sequence $\uu_i$ and apply Lemma  \ref{auxiliary lemma}, then
\[
\lim_{x_i\ra x_0}\frac{u^j(x_i)}{|\uu(x_i)|}=\lim_{i\ra\infty}\frac{u^j_i(0)}{|\uu_i(0)|}=\frac{u^j_0(0)}{|\uu_0(0)|}=\ee(x_0)\cdot\ee_j.
\]
%%%%%%%%%%%%%%%

{\it (iii)}
This statement follows from the fact that $\frac{u^1}{|\uu|}$ is continuous and its value at $x=0$ is equals one.

%%%%%%%%%%%%%%

{\it(iv)} For a sequence $x_l\ra x_0\in \cR_\uu$, consider the corresponding function $\uu_l(x):=\uu(x_l+d_lx)/d_l^\kp$ as defined in Lemma \ref{auxiliary lemma} converging to $\alpha((x+\nu(x_0))\cdot\nu(x_0))_+^\kp\ee(x_0)$. 
Thus %$u^1_i\ra \alpha((x+\nu(x_0))\cdot\nu(x_0))_+^\kp(\ee(x_0)\cdot\ee_1)$, and
we can write (notice that $|\uu_l(0)|\ra \alpha$ and we can assume that $|\uu_l(0)|\ne0$ for sufficiently large $l$)
\begin{align*}
\lim_{x_l\ra x_0}|\uu(x_l)|^{-\frac{1+q}2}\partial_{i}\uu(x_l)=&\lim_{l\ra \infty}|\uu_l(0)|^{-\frac{1+q}2}\partial_{i}\uu_l(0)
=\alpha^{-\frac{1+q}2}\kp\alpha (\nu(x_0)\cdot\ee_i)\ee(x_0).
%\lim_{x_i\ra x_0}\nabla(u^1)^{1/\kp}(x_i)=& \lim_{x_i\ra x_0}\frac1\kp(u^1(x_i))^{-\frac{1+q}2}\nabla u^1(x_i)\\
%=&\lim_{i\ra \infty}\frac1\kp(u_i^1(0))^{-\frac{1+q}2}\nabla u_i^1(0)\\
%=&\ \alpha^{1/\kp}\nu(x_0)(\ee(x_0)\cdot\ee_1)^{1/\kp},
\end{align*}
%where we have assumed that $x_0$ is close to $x=0$ such that $\ee(x_0)\cdot\ee_1\neq0$.
In addition,
\[\lim_{\{|\uu|>0\}\ni x\ra x_0}|\uu(x)|^{-\frac{1+q}2}\nabla u^1(x)=\kp\alpha^{1/\kp}\nu(x_0)(\ee(x_0)\cdot\ee_1),\]
which along with $(ii)$ implies that $\nabla (u^1)^{1/\kp}=\frac1\kp(u^1)^{-\frac{1+q}2}\nabla u^1$ is continuous.

%%%%%%%%%%%%%%

{\it(v)}  Proof is similar to previous parts.
\end{proof}

%%%%%%%%%%%%%%%%%%%%%%%
%%%%%%%%SECTION%%%%%%%%%%
%%%%%%%%%%%%%%%%%%%%%%%
\section{Partial Hodograph-Legendre transform}\label{Section:Hodograph}

In this section we provide the main tool of our approach, the change of coordinates involving dependent and independent variables. We begin with introducing the transform
\begin{equation}\label{Hodograph-definition}
T:x\mapsto y=(x_1,\cdots,x_{n-1},(u^1(x)/\alpha)^{1/\kp}).
\end{equation}
In order to show that the transform is well-defined, we need that $u^1\geq0$ in a neighborhood of $x=0$ which has been shown in  Proposition \ref{property-near-regular}.
The following proposition shows that $T$ is a local $C^{1}$-diffeomorphism.

\begin{proposition}\label{diffeomorphism}
Let    $0\in\cR_{\uu}$ and suppose    \eqref{blowup-assumption} holds. Then the transform $T$, defined in \eqref{Hodograph-definition}, is a $C^{1}$   diffeomorphism in a neighborhood of $x=0$ in $\overline{\{|\uu|>0\}}$.
\end{proposition} 
\begin{proof}
It easily follows from Proposition \ref{property-near-regular} part $(iv)$ that $\nabla(u^1/\alpha)^{1/\kp}$ is continuous in $\overline{\{|\uu|>0\}}$ and $\nabla(u^1/\alpha)^{1/\kp}(0)=\ee_n$. Therefore, $DT(0)=I$ is identity and invertible. 
The proof can now be  completed by invoking the inverse function theorem.
\end{proof}

We consider the inverse transform $x=T^{-1}(y)$ and define the partial Legendre transform of $\uu$ by
\[
v^1(y):=x_n-y_n,\quad v^j(y):=\frac{\alpha u^j(x)}{u^1(x)}\text{ for }2\leq j\leq m.
\] 
Since $T^{-1}$ is $C^1$, the Legendre function $\vv=(v^1,\dots, v^m)$ is $C^1$. 

Also, from relation $(u^1(x)/\alpha)^\kp=y_n$, we get
\[ u^1(T^{-1}(y))= u^1(y_1,y_2, \dots , y_{n-1}, y_n+v^1(y)) = \alpha y_n^\kp. \] 
As a consequence, the free boundary $\Gamma_{\uu}$ can be parametrized in a neighborhood of $x=0$ as
\[
\Gamma_{\uu}\cap B_\rho:=\{y\in B_\rho: y_n=v^1(y',0)\}.
\]
We are going  to show that 
\[ y'\mapsto v^1(y',0)  \]  
is analytic to prove the analyticity of free boundary. In addition, we will see that 
\[ y\mapsto \vv(y)  \] 
is also analytic  in $B_r^+(0)$ for some $r>0 $ to prove that the solution is analytic.

For the sake of simplicity let
\[
\tilde u^1(x):=(u^1(x)/\alpha)^{{1}/\kp}=y_n, \quad w^j(y):=y_n^{\kp}\,v^j(y)=u^j(x) \text{ for }2\leq j\leq m.
\]
Then 
\[    v^1(x',  \tilde u^1(x)) = x_n - \tilde u^1(x), \quad  w^j(x',\tilde u^1(x))=u^j(x),\] 
and the implicit function theorem implies
\begin{equation}\label{relation-derivatives-legendre}
\begin{aligned}
 \nabla_x \tilde u^1 = \frac1{1+\partial_{y_n} v^1}  \left( \begin{matrix} -\nabla'_y v^1
\\ 1 \end{matrix} \right), & \quad
\nabla_y w^j=\left( \begin{matrix} \nabla'_x u^j \\ 0\end{matrix} \right)+\frac{\partial_{x_n}u^j}{\partial_{x_n}\tilde u^1 }\left( \begin{matrix} -\nabla'_x \tilde u^1 \\ 1\end{matrix} \right),
\ \text{ for }j>1, 
\end{aligned}\end{equation}
where $\nabla'=(\partial_{1},\cdots,\partial_{{n-1}})$.
%\medskip
%\nabla_x u^j=\left( \begin{matrix} \nabla'_y w^j \\ 0\end{matrix} \right)+\partial_{y_n}w^j\nabla_x \tilde u^1,
%\nabla_y v^1 = \frac{1}{\partial_{x_n}\tilde u^1 }\left( \begin{matrix} -\nabla'_x \tilde u^1 \\ 1-\partial_n \tilde u^1\end{matrix} \right), & \quad

\begin{lemma}\label{lem-vanishing-v-derivative}
The Legendre function $\vv=(v^1,\cdots,v^m)$ is a $C^{1}$ function in $B_r^+(0)$ for some $r>0$ and satisfies 
\begin{enumerate}[(i)]
\item
$\nabla v^1(0)=0$. 

\item
For $1\leq j\leq n-1$ and $1\leq i\leq n$, we have
\[
\lim_{y\ra0}y_n\partial_{ij}v^1(y)=0.
\]

\item
$y_n\partial_{nn}v^1(y)$ is bounded near $y=0$.

\item
For $2\leq j\leq m$, we have $v^j(y)=o(1)$ as $y\ra0$. Also, $\nabla v^j(y)=o(y_n^{-1})$ as well as $D^2 v^j(y)=o(y_n^{-2})$.
\end{enumerate}
\end{lemma}

\begin{proof}
{\it (i)} 
Since $T$ is a $C^1$ transform, so $\vv$ will be $C^1$ as well.  Let $x=0$ in \eqref{relation-derivatives-legendre} and invoke Proposition \ref{property-near-regular} part $(iv)$ to get $\nabla_x\tilde u^1(0)=\ee_n$, it implies that $\nabla v^1(0)=0$. 

{\it (ii)} If $1\leq i,j\leq n-1$, by implicit function theorem we have 
\[
\partial_{y_iy_j}v^1=-\frac{\partial_{x_ix_j}\tilde u^1}{\partial_{x_n}\tilde u^1}+\frac{\partial_{x_ix_n}\tilde u^1\partial_{x_j}\tilde u^1+\partial_{x_jx_n}\tilde u^1\partial_{x_i}\tilde u^1}{(\partial_{x_n}\tilde u^1)^2}-\frac{\partial_{x_nx_n}\tilde u^1\partial_{x_i}\tilde u^1\partial_{x_j}\tilde u^1}{(\partial_{x_n}\tilde u^1)^3} .
\]
On the other hand, 
\[
\tilde u^1\partial_{k\ell}\tilde u^1=\frac1\kp(\frac1\kp-1)(u^1)^{-1-q}\partial_ku^1\partial_\ell u^1+\frac1\kp(u^1)^{-q}\partial_{k\ell}u^1,
\]
vanishes at $x=0$ by Proposition \ref{property-near-regular}, if $(k,\ell)\neq(n,n)$. Also, it is bounded near $x=0$ when  $k=\ell=n$.
It yields that $y_n\partial_{y_iy_j}v^1\ra0$ when $y\ra0$.

If $1\leq j\leq n-1$, we have 
\[
\partial_{y_jy_n}v^1=\frac{\partial_{x_nx_n}\tilde u^1\partial_{x_j}\tilde u^1-\partial_{x_jx_n}\tilde u^1\partial_{x_n}\tilde u^1}{(\partial_{x_n}\tilde u^1)^3},
\]
and similar argument shows that $y_n\partial_{y_jy_n}v^1$ vanishes at $y=0$.

{\it (iii)} The statement can be concluded easily by Proposition \ref{property-near-regular} part $(v)$ and the following relation:
\[
\partial_{y_ny_n}v^1=-\frac{\partial_{x_nx_n}\tilde u^1}{(\partial_{x_n}\tilde u^1)^3}.
\]

{\it (iv)} 
It is more convenient to work with $w^j$.
Part $(ii)$ in Proposition \ref{property-near-regular} shows that 
$$w^j(y)=u^j(x)=o(|\uu(x)|)=o(u^1(x))=o(y_n^\kp), \, \text{ as }y\ra0.$$
Furthermore, from part $(iv)$ in Proposition \ref{property-near-regular}, we know that 
\[
\nabla u^j(x)=o(|\uu(x)|^{\frac{1+q}2})=o((u^1(x))^{\frac{1+q}2})=o(y_n^{\kp-1}), \, \text{ as }x\ra0.
\]
This together with \eqref{relation-derivatives-legendre} proves that $\nabla w^j(y)=o(y_n^{\kp-1})$. 
This shows the desired property for the behavior of $\nabla v^j$ near to $y=0$.

To obtain the growth of second derivative of $w^j$, first consider   $1\leq k,\ell\leq n-1$ and 
\[
\partial_{y_ky_\ell}w^j=\partial_{x_kx_\ell}u^j+\partial_{x_kx_n}u^j\partial_{y_\ell} v^1+\partial_{x_\ell x_n}u^j\partial_{y_k}v^1+\partial_{x_nx_n} u^j \partial_{y_k} v^1\partial_{y_\ell} v^1+\partial_{x_n} u^j\partial_{y_k y_\ell}v^1. 
\]
Now apply the result in $(ii)$, as well as parts $(ii)$,  $(iv)$ and $(v)$ in Proposition \ref{property-near-regular},  we get
\[
\partial_{y_ky_\ell}w^j(y)=o(|\uu(x)|^q)+o(y_n^{-1}|\uu(x)|^{\frac{1+q}2})=o(y_n^{\kp-2}).
\]

By a similar argument the statement holds for partial derivatives
\[
\partial_{y_ny_\ell}w^j=\partial_{x_nx_\ell}u^j\partial_{y_n}v^1+\partial_{x_nx_n}u^j\partial_{y_\ell} v^1\partial_{y_n}v^1+\partial_{x_n} u^j\partial_{y_n y_\ell}v^1,\quad\text{ for }1\leq \ell\leq n-1,
\]
and also
\[
\partial_{y_ny_n}w^j=\partial_{x_nx_n}u^j(\partial_{y_n}v^1)^2+\partial_{x_n}u^j\partial_{y_ny_n}v^1.
\]
Therefore, $D^2w^j(y)=o(y_n^{\kp-2})$ and according to previous results for $v^j$ and $\nabla v^j$ we obtain that $D^2v^j(y)=o(y_n^{-2})$.
\end{proof}

%\bigskip

In the sequel we will obtain a nonlinear PDE system for the Legendre vectorial  function $\vv$. 
\begin{align*}
\partial_{x_i}^2\tilde u^1=&-\partial_{y_i}\left(\frac{\partial_{y_i}v^1}{1+\partial_{y_n}v^1}\right)+\frac{\partial_{y_i}v^1}{1+\partial_{y_n}v^1}\partial_{y_n}\left(\frac{\partial_{y_i}v^1}{1+\partial_{y_n}v^1}\right)\\[8pt]
=&-\frac{\partial_{y_iy_i}v^1}{1+\partial_{y_n}v^1}+\frac{\partial_{y_n}(\partial_{y_i}v^1)^2}{(1+\partial_{y_n}v^1)^2}-\frac{(\partial_{y_i}v^1)^2\partial_{y_ny_n}v^1}{(1+\partial_{y_n}v^1)^3},
\end{align*}
for $i=1,\cdots,n-1$ and 
\begin{align*}
\partial_{x_n}^2\tilde u^1=&\frac{1}{1+\partial_{y_n}v^1}\partial_{y_n}\left(\frac{1}{1+\partial_{y_n}v^1}\right).
\end{align*}
Thus
\[ \Delta_x \tilde u^1 = - \frac1{1+\partial_{y_n} v^1} \Big( \Delta v^1 - \partial_{y_n} \frac{ |\nabla v^1|^2}{1+\partial_{y_n} v^1 }\Big).   \] 
On the other hand, 
\[
\begin{split} 
\Delta_x u^1 \, & = \alpha \Delta_x (\tilde u^1)^{\kp} 
\\ & =  \alpha  \kp  (\tilde u^1)^{\kp-1} \Delta_x \tilde u^1
+ \alpha \kp(\kp-1) (\tilde u^1)^{\kp-2} |\nabla \tilde u^1|^2
\\&=- \frac{\alpha  \kp y_n^{\kp-1}}{1+\partial_{y_n} v^1} \Big( \Delta v^1 - \partial_{y_n} \frac{|\nabla v^1|^2}{1+\partial_{y_n} v^1 }\Big)
+\alpha \kp(\kp-1) y_n^{\kp-2}\frac{1+ |\nabla' v^1|^2}{(1+\partial_{y_n} v^1)^2 }
\end{split} 
\] 
taking into account that $\Delta_x u^1=|\uu|^{q-1}u^1=\alpha^{q} y_n^{\kp-2}|\vv'|_*^{q-1}$,
where 
\begin{equation*}%\label{norm-s}
|\vv'|_*=(1+ |\vv'/\alpha|^2)^{1/2}=\left(1+(v^2/\alpha)^2+\cdots+(v^m/\alpha)^2\right)^{1/2}.
\end{equation*}
Finally, we get 
\begin{equation}\label{equ:v1}
\begin{split}
y_n & \left(\Delta v^1 - \partial_{y_n} \frac{|\nabla v^1|^2}{1+\partial_{y_n} v^1 }\right)+ 2(\kp-1) \partial_{y_n} v^1 \\
& -(\kp-1) \frac{|\nabla v^1|^2}{1+\partial_{y_n} v^1}
- (\kp-1)(1-  |\vv'|_*^{q-1}) (1+ \partial_{y_n} v^1)=0.  
\end{split}
\end{equation}
%{\color{red} 
%\begin{equation} 
%y_n\left(\Delta v^1 - \partial_{y_n} \frac{|\nabla v^1|^2}{1+\partial_{y_n} v^1 }\right)+2(\kp-1) \partial_{y_n} v^1 -(\kp-1) \frac{(1-|\vv'|_*^{q-1})  + |\nabla' v^1|^2
%+  |\vv'|_*^{q-1}  (\partial_{y_n} v^1)^2} {1+\partial_{y_n} v^1}=0.  
%\end{equation}
%} 
Define the degenerate operator 
\[
\cL:=y_n\left(\Delta + \sum_{\ell=1}^n a_\ell \partial_{\ell n}\right),
\]
where 
\begin{equation}\label{coefficient-eq}
 a_\ell = \left\{  \begin{array}{cl}  -\frac{2\partial_\ell v^1}{1+\partial_n v^1} & \text{ if } \ell <n \\ [8pt]
-\frac{2\partial_n v^1}{1+\partial_n v^1} + \frac{|\nabla v^1|^2}{(1+\partial_n v^1)^2} & \text{ if } \ell=n 
\end{array} \right.     
\end{equation}
Then we can rewrite \eqref{equ:v1}  as an degenerate elliptic equation
\[
\cL v^1 + 2(\kp-1)\partial_n v^1 =b,
\]
%{\color{red} 
%\[
%\cL v^1+2(\kappa-1) \partial_n v^1 :=y_n (\Delta v + \sum_{k=1}^n c_{k} \partial_{kn} v^1) + 2(\kappa-1) \partial_n v^1 = b  ,
%\]
%where 
%
%and 
\begin{equation} \label{eq:eqb} 
b =  (\kp-1) \frac{  |\nabla v^1|^2
+  (1-|\vv'|_*^{q-1})  (1 + \partial_{y_n} v^1)^2} {1+\partial_{y_n} v^1}.
\end{equation}
Observe that the $a_\ell$ vanishes if $v^1=0$ and 
\[ b = O(|\nabla v^1|^2 +  |\vv'|^2).     \]
Here, $|\vv'|$ is the standard Euclidean norm of vector $\vv'$.
% } 

%where
%\begin{equation}\label{coefficient-eq}
%\begin{split}
%a_{k\ell}&=\delta_{k\ell},\text{ if }k,\ell<n,\\
%a_{n\ell}&=a_{\ell n}=-\frac{\partial_\ell v^1}{1+\partial_n v^1},\text{ if }\ell<n,\\
%a_{nn}&=\frac{1+ |\nabla' v^1|^2}{(1+\partial_n v^1)^2},\\
%b&=\frac{1}{\kp}  |\vv'|_*^{q-1}-(\kp-1) \frac{1+ |\nabla' v^1|^2}{(1+\partial_{n} v^1)^2}.
%\end{split}
%\end{equation}

Also for $j\geq 2$, the following equation will be obtained
\begin{align*}%\label{equ:vj}
\partial_{x_i}^2 u^j = & \partial_{y_i}\left(\partial_{y_i} w^j - \frac{\partial_{y_n} w^j \partial_{y_i} v^1}{1+\partial_{y_n} v^1}\right)
- \partial_{y_n}\left(\partial_{y_i} w^j - \frac{\partial_{y_n} w^j \partial_{y_i} v^1}{1+\partial_{y_n} v^1}\right)\times\frac{ \partial_{y_i} v^1}{1+\partial_{y_n} v^1}\\
\partial_{x_n}^2 u^j = &\partial_{y_n}\left(\frac{ \partial_{y_n} w^j}{1+\partial_{y_n} v^1} \right)\times \frac{ 1}{1+\partial_{y_n} v^1} 
\end{align*}
and so
\begin{align*}
\Delta_x u^j=&\Delta' w^j+(\partial_{y_n}^2w^j)\frac{1+ |\nabla' v^1|^2}{(1+\partial_{y_n} v^1)^2 }-\frac{ 2\nabla' v^1\cdot\nabla'\partial_{y_n}w^j}{1+\partial_{y_n} v^1} \\
&- \frac{\partial_{y_n}w^j}{1+\partial_{y_n} v^1} \Big( \Delta v^1 - \partial_{y_n} \frac{ |\nabla v^1|^2}{1+\partial_{y_n} v^1 }\Big),\\ 
& = y_n^{-1}\Big( \mathcal{L} w^j -\Big(b - 2( \kappa-1) \partial_{y_n} v^1  \Big)     
\frac{\partial_{y_n} w^j}{1+ \partial_{y_n} v^1}     \Big)
\end{align*}
where $\Delta'=\partial_{y_1}^2+\cdots+\partial_{y_{n-1}}^2$.
Recall $\Delta_x u^j= |\uu|^{q-1}u^j= \alpha^{q-1}y_n^{-2}|\vv'|_*^{q-1}w^j $,  after simplifying 
\begin{align}\label{eq:wj}
\cL w^j-\Big( b - 2( \kappa-1) \partial_{y_n} v^1  \Big)     
\frac{\partial_{y_n} w^j}{1+ \partial_{y_n} v^1}     \Big)=  \kappa (\kappa-1) y_n^{-1} w^j|\vv'|_*^{q-1}.
\end{align}
Put $w^j=y_n^\kappa v^j$ in this equation, we will get
\[\begin{split}
\cL w^j = & y_n^{\kp}\cL v^j 
+ \kp y_n^{\kp} \sum_{\ell} a_{\ell } \partial_\ell v^j 
 + \kp y_n^{\kp}( a_{n}+2) \partial_n v^j 
 +\kp (\kp-1)(a_{n}+1) y_n^{\kp-1} v^j.
\end{split}\]
This together with \eqref{eq:wj} gives
\[\begin{split}
\cL v^j + \kp \sum_{\ell} a_{\ell } \partial_\ell v^j 
 + \kp( a_{n}+2) \partial_n v^j 
 -\Big( b - 2( \kappa-1) \partial_{n} v^1  \Big)    
 \frac{\partial_{n} v^j}{1+ \partial_{n} v^1} \Big) =0.
\end{split}\]

By notation 
\begin{equation}\label{eq:eqc} 
c= \kp a_n - \frac{b - 2( \kappa-1) \partial_{n} v^1 }{1+ \partial_{n} v^1},
\end{equation}
we come up with the following proposition:

\begin{proposition}
The Legendre function  $\vv=(v^1,\cdots,v^m)$ satisfies the following system
\begin{equation}\label{degenerate-eq}
\begin{split}
&\cL v^1 + 2(\kp-1)\partial_n v^1 =b,\\
& \cL v^j+ 2\kp \partial_n v^j +
\kp\sum_{\ell=1}^n a_{\ell }\partial_\ell v^j+c \partial_n  v^j=0,\  \text{ for }j>1,
\end{split}
\end{equation}
where the coefficients $a_{\ell}$ depend analytically on $ \nabla  v^1$ and vanish if $ \nabla v^1=0$, 
$b$ depends analytically  on $ \nabla v^1$  and $v^j$, $ j \ge 2$,  and is  $ O( |\nabla v^1 |^2+ |\vv' |^2)$. 
In addition, $c$ is an analytic function of $ \nabla v^1$ and $b$, and it vanishes if $ \nabla v^1=0$ and $b=0$.
They are given in \eqref{coefficient-eq}, \eqref{eq:eqb} and \eqref{eq:eqc}. 
\end{proposition}

\begin{remark}
Using the scaling $\uu(rx)/r^\kp$ and the result of Lemma \ref{lem-vanishing-v-derivative}  in addition to the invariance of equations (\ref{degenerate-eq}) under the scaling
\[
\vv_r(y)=(\frac{v^1(ry)}r,v^2(ry),\dots,v^m(ry)),
\]
we may hence without loss of generality suppose that 
\begin{equation}\label{assumption-v-small}
\begin{split}
\norm{v^1}_{C^{1}(B_2^+)}+\norm{y_n D^2v^1}_{L^\infty(B_2^+)}\leq \e_0,\\
\norm{v^j}_{L^{\infty}(B_2^+)}+\norm{y_n\nabla v^j}_{L^{\infty}(B_2^+)}+\norm{y_n^{2}D^2v^j}_{L^{\infty}(B_2^+)}\leq \e_0,
\end{split}
\end{equation}
for any small value of $\e_0$, which is specifically determined in Proposition \ref{prop:v-W3p} and Section \ref{section:main-result}. 
This assumption implies that 
\[
1 \leq |\vv'|_* \leq 1 + (m-1)^{1/2}\e_0.
\]
In particular, $|\vv'|_*^{q-1}\approx 1$.
\end{remark}

%\medskip

%%%%%%%%%%%%%%%%%%%%%%%
%%%%%%%%SECTION%%%%%%%%%%
%%%%%%%%%%%%%%%%%%%%%%%

\section{Regularity in a degenerate equation}\label{section:degenerate}
The key point in our approach is that  under assumption \eqref{assumption-v-small} %the solution $\vv$ is close to $\tilde\vv=(y_n,0,\cdots,0)$ and 
the system \eqref{degenerate-eq} is a perturbation of operator $y_n\Delta  + \gamma \partial_n $ for some $\gamma>0$.
We will show in this setting the solution of perturbed equation will be analytic as well. 
In this section we first review the regularity of solutions for degenerate equations $y_n\Delta u + \gamma \partial_n u=f$.
Similar results can be found in some literature. We, however,  provide a short background  for readers' convenience.

We start with the definition of a quotient Banach space
\[
W^{k,p}_*(\Omega):=\{ u: \nabla u \in W^{k-2,p}(\Omega)\text{ and } x_nD^ku\in L^p(\Omega)\}/\bR,
\]
which is equipped with norm
\[
\|u\|_{W^{k,p}_*(\Omega)}:=\|\nabla u\|_{W^{k-2,p}(\Omega)}+\|x_nD^ku\|_{L^p(\Omega)}.
\]

Here is the main result for the regularity of the degenerate equation in the clean case. 

\begin{theorem}\label{reg-deg-lap}
For   $\gamma>0$, and for  every $f\in L^p(\bR^n_+)$  the equation
\begin{equation}\label{deg-lap}
\Delta_\gamma u:=y_n\Delta u+\gamma\partial_nu=f,\qquad\text{ in }\bR^n_+,
\end{equation}
has  a unique (up to a constant) solution  $ u\in W^{2,p}_*(\bR^n_+)$ and the following estimate holds
\[
\norm{ u}_{W^{2,p}_*(\bR^n_+)}\leq C_{n, p, \gamma}\norm{f}_{L^p(\bR^n_+)},
\] 
where $C_{n,p,\gamma}$ is a universal constant.
In addition,  there is constant $\tilde C_{n, p, \gamma}$ such that
if $f\in W^{1,p}(\bR^n_+)$, then $ u\in W^{3,p}_*(\bR^n_+)$ and
\[
\| u\|_{W^{3,p}_*(\bR^n_+)}\leq \tilde C_{n, p, \gamma}\| f\|_{W^{1,p}(\bR^n_+)}.
\]
\end{theorem}

%\medskip

Similar results for parabolic version of the  degenerate operator $\Delta_\gamma-\partial_t$  has been proved earlier; see  Theorem 2 in \cite{kochnon}. 
We shall, however, provide a relatively short and self-contained proof for our case in Appendix.

We apply this theory of degenerate PDEs to improve   a  priori  regularity of Legendre functions $\vv=(v^1,\cdots,v^m)$, which is also a solution of \eqref{degenerate-eq}. 
From Lemma \ref{lem-vanishing-v-derivative}, we already know that $v^1\in W^{2,\infty}_*(B_1^+)$ as well as $y_n v^j\in W^{2,\infty}_*(B_1^+)$ for $j>1$.

\begin{proposition}\label{prop:v-W3p}
Suppose that $\vv$ is a solution of system \eqref{degenerate-eq} in $B_2^+$ and satisfies \eqref{assumption-v-small}.
If $\e_0$ in \eqref{assumption-v-small} is small enough, then $\vv\in W^{3,p}_*(B_{1}^+;\bR^m)$ and
\begin{equation}\label{v-W3p}
\norm{\vv}_{W^{3,p}_*(B_{1}^+)}\leq C\e_0,
\end{equation}
where constant $C$ depends only on  $m, n, p$ and $\kp$.
\end{proposition}
\begin{proof}
{\bf Step 1 -} $W^{2,p}_*$-\,estimate for $v^j$ $(j>1)$:\\

Choose a cutoff function $\zeta$ such that $\zeta\equiv 1$ on $B_{9/5}^+$ and $\eta\equiv0$ on $\bR^n_+\setminus B_{2}^+$, 
and set $ \bar v^j = \zeta  v^j$ which solves the problem
\begin{align}
y_n\Delta  \bar v^j+2\kp\partial_n  \bar v^j=&
-\sum_{\ell}a_{\ell }\left(\kp\partial_\ell\bar v^j+y_n\partial_{\ell n}\bar v^j\right)
- c\partial_{n}\bar v^j\notag\\
&+y_n\sum_{\ell}a_{\ell }\left(\partial_\ell\zeta \partial_n v^j + \partial_n \zeta \partial_\ell v^j\right)
+2y_n\nabla\zeta\cdot\nabla v^j
\notag\\
&+\left(\cL\zeta+2\kp\partial_n\zeta +\kp \sum_{\ell} a_{\ell } \partial_\ell\zeta  + c \partial_n\zeta\right) v^j
\label{prop:equ:barvj} .
\end{align}
Define the operator 
\[
\cG^j u:=\sum_{\ell}a_{\ell }\left(\kp\partial_\ell u+y_n\partial_{\ell n} u\right)
+ c\partial_{n}u.
%-2\sum_{\ell<n}a_{\ell n}\left(\kp\partial_\ell u+y_n\partial_{\ell n}u\right)
%+y_n(1-a_{nn})\partial_{nn}u
%+(2\kp-2\kp a_{nn}-b)\partial_n u.
\]
According to \eqref{assumption-v-small} and  the definition of  $a_\ell$ and $c$, the norm of  the operator $\cG^j: W^{2,p}_*(\bR^n_+)\ra L^p(\bR^n_+)$ satisfies 
\[
\norm{\cG^j}\leq C\e_0.
\]
Therefore, when $\e_0$ is small enough the operator $\Delta_{2\kp} + \cG^j$ is invertible, implying 
\[
\norm{\bar v^j}_{W^{2,p}_*(\bR^n_+)}\leq 2C_{n, p, 2\kp}\norm{(\Delta_{2\kp} + \cG^j)\bar v^j}_{L^{p}(\bR^n_+)},
\] 
where the constant $C_{n, p, 2\kp}$ is defined in Theorem \ref{reg-deg-lap}. Here  we have assumed that $\e_0$ is small enough such that $\norm{\Delta_{2\kp}^{-1}\cG^j}\leq 1/2$.
This together with \eqref{prop:equ:barvj} and \eqref{assumption-v-small}  gives
\begin{align*}
\norm{\bar v^j}_{W^{2,p}_*(B_{2}^+)}\leq C\e_0,
\end{align*}
for a universal constant $C$.
Finally, we have 
\[
\norm{ v^j}_{W^{2,p}_*(B_{9/5}^+)}\leq C\e_0.
\]

\medskip

%%%%%%%%%%%%%%%%STEP 2%%%%%%%%%%%%%%

{\bf Step 2 -} $W^{3,p}_*$-\,estimate for $ v^1$:\\

Differentiate equation \eqref{degenerate-eq} with respect to a tangential direction $\ee_k$, $1\leq k<n$
\begin{align}\label{eq:v_k^1}
y_n \Delta v_k^1 + 2(\kp-1)\partial_n v_k^1=& -\sum_{\ell} y_n a_\ell \partial_{\ell n} v_k^1 - \sum_{i,\ell=1}^n y_n \partial_{v^1_i}a_\ell \partial_{i} v_k^1 \partial_{\ell n}v^1 \\
&+\sum_{j=2}^m \partial_{v^j}b \partial_k v^j + \sum_{i=1}^n \partial_{v^1_i}b \partial_i v^1_k, \notag
\end{align}
where $v_k^1:=\partial_k  v^1$,  $\partial_{v^1_i}$ is derivative with respect to $\partial_i v^1$ when $a_\ell(\nabla v^1)$ or $b(\nabla v^1,\vv')$ are considered as  functions of $\nabla v^1$. Similarly,  $\partial_{v^j}$ denotes the derivative with respect to $v^j$ when  $b$ is considered as a function of $\vv'$.
Let us consider a cutoff function $\eta$ such that $\eta\equiv 1$ on $B_{8/5}^+$ and $\eta\equiv0$ on $\bR^n_+\setminus B_{9/5}^+$, 
and set $ \hat v^1 = \eta v_k^1$. We have
\begin{align*}
y_n \Delta\hat v^1 +&2(\kp-1)\partial_n\hat v^1  =\cG^1(\hat v^1) + I_1+I_2,
\end{align*}
where
\begin{align*}
\cG^1(\hat v^1):=&-\sum_{\ell} y_n a_\ell \partial_{\ell n} \hat v^1  - \sum_{i,\ell=1}^n (y_n \partial_{\ell n}v^1) \partial_{v^1_i}a_\ell \partial_{i} \hat v^1 
+  \sum_{i=1}^n \partial_{v^1_i}b \partial_i \hat v^1, \\
I_1:=& \sum_{\ell} y_n a_\ell  \left( \partial_{\ell n} \eta v_k^1 + \partial_\ell \eta \partial_n v_k^1+\partial_n\eta \partial_\ell v_k^1\right)
+ \sum_{i,\ell=1}^n (y_n \partial_{\ell n}v^1) \partial_{v^1_i}a_\ell \partial_{i} \eta v_k^1 \\
&-  \sum_{i=1}^n \partial_{v^1_i}b \partial_i \eta v_k^1 
+ (y_n\Delta \eta + 2(\kp-1)\partial_n\eta) v^1_k+ 2y_n\nabla\eta\cdot\nabla v^1_k,\\
I_2:=&\sum_{j=2}^m \eta \partial_{v^j}b \partial_k v^j . 
\end{align*}
The norm of operator $\cG^1: W^{2,p}_*(\bR^n_+)\ra L^p(\bR^n_+)$  is  controlled by  $\e_0$ due to the assumption  \eqref{assumption-v-small} and definition of  $a_\ell$ and $b$.
 Similar to Step 1, the operator $\Delta_\gamma-\cG^1$ for $\gamma=2(\kp-1)$ is invertible and 
 \[
\norm{\hat v^1}_{W^{2,p}_*(\bR^n_+)}
\leq 2C_{n, p, \gamma}\norm{I_1+I_2}_{L^p(\bR^n_+)}.
\] 
The $L^p$-norm of $I_1$  is controlled by $\norm{  v^1}_{W^{2,p}_*(B_2^+)}$ and so by $\e_0$. 
The term $I_2$ can  also be estimated by  $\norm{\nabla \vv'}_{L^p(B_{9/5}^+)}$ since $\partial_{v^1_i} a_\ell $ and $\partial_{v^j}b $ are bounded.
All these together with the result of Step 1, that $\|\partial_k v^j\|_{L^p(B_{9/5}^+)}\leq C\e_0$, give
\begin{equation}\label{estimate-tangantial-direction}
\norm{\partial_k   v^1}_{W^{2,p}_*(B_{8/5}^+)}\leq C\e_0,
\end{equation} 
for a universal constant $C$ depending only on $n, p, \kp$.

Now repeat the above argument  for  derivative  with respect to  direction $\ee_n$. 
Notice that in this case ($k=n$), we shall add the following term in right hand side of \eqref{eq:v_k^1}
\[
-\Delta v^1 - \sum_{\ell}a_{\ell}\partial_{\ell n}  v^1,
\]
 where all terms except $\partial_{n n}  v^1$ can be controlled by the estimate \eqref{estimate-tangantial-direction}. 
 We  overcome this challenge by adding the term $\partial_n v_n^1$ on both sides, and continue other calculations with the operator 
 $y_n \Delta v_n^1 + (2\kp-1)\partial_n v_n^1$. 
 Finally we can get
\[
\norm{\partial_n   v^1}_{W^{2,p}_*(B_{3/2}^+)}\leq C\e_0.
\]

%%%%%%%%%%%%%%%%STEP 3%%%%%%%%%%%%%%

{\bf Step 3 -} $W^{3,p}_*$-\,estimate for $ v^j$ ($j>1$):\\

Let $v_k^j:=\partial_k v^j$ for a tangential direction $\ee_k$, $1\leq k<n$ and differentiate equation \eqref{degenerate-eq}  to get
\begin{align*}
y_n\Delta  v^j_k+2\kp\partial_n  v^j_k=&-\sum_{\ell} y_n a_\ell \partial_{\ell n} v_k^j -\sum_\ell \kp a_\ell \partial_\ell v_k^j - c \partial_n v_k^j \\
&- \sum_{i,\ell}  \partial_{v^1_i}a_\ell \partial_{ik} v^1( y_n \partial_{\ell n}v^j + \kp \partial_{\ell }v^j) \\
&-\left(\sum_{i=1}^n \partial_{v^1_i} c ~ \partial_{ik} v^1 + \sum_{s=2}^m \partial_{v^s} c ~ \partial_{k} v^s \right)\partial_n v^j
\end{align*}
Let $\eta$ be a cutoff function such that $\eta\equiv 1$ on $B_{5/4}^+$ and $\eta\equiv0$ on $\bR^n_+\setminus B_{3/2}^+$. 
Now let $ \hat v^j = \eta v^j_k$, then
\begin{align*}
y_n & \Delta  \hat v^j+2\kp\partial_n  \hat v^j + \cG^j(\hat v^j) = I_1 + I_2,
\end{align*}
where $\cG^j$ is the operator defined in Step 1 and 
\begin{align*}
I_1 =&  \sum_{\ell} y_n a_\ell (\partial_{\ell n} \eta v_k^j + \partial_\ell \eta \partial_n v_k^j + \partial_n \eta \partial_\ell v_k^j)
+ \left(\sum_\ell \kp a_\ell \partial_\ell \eta  + c \partial_n \eta \right) v_k^j\\
& +(y_n\Delta \eta +2\kp \partial_n\eta)v_k^j + 2y_n\nabla\eta\cdot\nabla v_k^j \\
I_2 = &- \sum_{i,\ell} \eta \partial_{v^1_i}a_\ell \partial_{ik} v^1( y_n \partial_{\ell n}v^j + \kp \partial_{\ell }v^j) 
-\left(\sum_{i=1}^n \partial_{v^1_i} c  \partial_{ik} v^1 + \sum_{s=2}^m \partial_{v^s} c  \partial_{k} v^s \right)\eta\partial_n v^j .
\end{align*}
%Repeat the argument in Step 2 for calculation $\eta \partial_k a_{\ell n}$ and $\partial_k b$ to obtain a perturbation of operator $\Delta_\kp$ and apply  Theorem \ref{reg-deg-lap}  to get 
We now apply  Theorem \ref{reg-deg-lap} to the  operator $\Delta_{2\kp}$, which is perturbed by operator $\cG^j$,
%. Similar to Step 1, $\Delta_\kp-\cG^t$ is invertible and 
\[
\norm{\hat v^j}_{W^{2,p}_*(\bR^n_+)}\leq 2C_{n, p, 2\kp}\norm{(\Delta_{2\kp} + \cG^j)\hat v^j}_{L^p(\bR^n_+)},
\] 
so,
\begin{align*}
\norm{I_1}_{L^p(\bR^n_+)}\leq & C\e_0 \norm{ v^j}_{W^{2,p}_*(B_{3/2}^+)} \\ 
\norm{I_2}_{L^p(\bR^n_+)}\leq & C\left(\norm{D^2v^1}_{L^{2p}(B_{3/2}^+)} +\norm{\nabla \vv'}_{L^{2p}(B_{3/2}^+)} \right) \norm{ v^j}_{W^{2,2p}_*(B_{3/2}^+)} .
\end{align*}
This together with the results of steps 1 and 2 for $2p$, imply 
\[
\norm{\partial_k  v^j}_{W^{2,p}_*(B_{5/4}^+)}\leq C\e_0,
\]
which can be used to repeat the argument to show  similar estimate for $\partial_n  v^j$.
\end{proof}

\medskip

%%%%%%%%%%%%%%%%%%%%%%%
%%%%%%%%SECTION%%%%%%%%%%
%%%%%%%%%%%%%%%%%%%%%%%

\section{Proof of Theorems \ref{main-result} and \ref{thm:existence}}\label{section:main-result}

\subsection{An infinitesimal translation}
We first introduce here a one-parameter family of diffeomorphisms which is obtained from the following ODE
\begin{equation}\label{ODE}
\begin{aligned}
\phi'(t)=&a((3/4)^2-|\phi(t)|^2)_+^3\eta(z_n),\\
\phi(0)=&z',
\end{aligned}
\end{equation}
where $a\in \bR^{n-1}$ and $z\in\bR^n$ are fixed.
Here, $\eta$ is a smooth cut-off function supported in $|z_n|\leq 1/2$ and $\eta\equiv1$ for $|z_n|\leq 1/4$.
We denote the unique solution to \eqref{ODE} by $\phi_{a,z}(t)$ and define the diffeomorphism 
\[
\Phi_a(z):=(\phi_{a,z}(1),z_n).
\]
The following is a classical result in the theory of ODEs for the regularity of this family of diffeomorphisms.

\begin{lemma}
For each $a\in\bR^{n-1}$ fixed, $\Phi_a:\bR^{n}\ra\bR^{n}$ is a $C^3$ diffeomorphism. Furthermore, the mapping $a\mapsto \Phi_a\in C^3(\bR^n)$ is analytic.
\end{lemma}

 The diffeomorphism $\Phi_a$ has other useful properties as well. 
 We just present the results here since the proof is a matter of computation and can be settled easily.
 
 \begin{lemma}
 \begin{enumerate}[(i)]
 \item For every $a\in\bR^{n-1}$, we have $\Phi_a(\{z_n=0\})\subset \{z_n=0\}$.
 In addition, $\Phi_a(z)=z$ when $z\notin \{z:|z'|\leq \frac34, |z_n|\leq \frac12\}$.
 
 \item For every $a\in\bR^{n-1}$, we have $\partial_{z_n}\Phi_a(z)=(0,\cdots,0,1)$ when $|z_n|\leq 1/4$.
 
 \item For $a=0$, we have $\Phi_0(z)=z$ for all $z\in\bR^n$.
 \end{enumerate}
 \end{lemma}

\begin{remark}\label{rmk:phi-diffeomor}
Indeed, $\psi_a(t)=\frac{\partial}{\partial a_i} \phi_{a,z}(t)$ satisfies 
\[
\psi'_a(t)=((3/4)^2-|\phi_{a,z}|^2)_+^3\eta(z_n)\ee_i-6a((3/4)^2-|z'|^2)_+^2\eta(z_n)(\phi_{a,z}\cdot\psi_a)\,,
\]
and for the derivative at $a=0$, the equation will simplify to 
\[
\psi_0'(t)=((3/4)^2-|z'|^2)_+^3\eta(z_n)\ee_i\,.
\]
This shows that the derivative of $a\mapsto \phi_{a,z}(1)$ is invertible at $a=0$, so for a fix point $z$ map $a\mapsto \phi_{a,z}(1)$ is an analytic diffeomorphism in a neighborhood of $a=0$ in $\bR^{n-1}$. We use this property later.
\end{remark}

Let $\vv$ be the Legendre vectorial function introduced in Section \ref{Section:Hodograph} and satisfy \eqref{degenerate-eq}.
Now we use the family of diffeomorphisms $\Phi_a$  to define a one-parameter family of  functions 
\[
\vv_a(z):= \vv (\Phi_a(z)).
\]
We  observe that the space $W^{3,p}_*$ is stable under the diffeomorphism $\Phi_a$.
\begin{lemma}
If $v\in W^{3,p}_*(B_1^+)$, then $v_a=v\circ\Phi_a\in W^{3,p}_*(B_1^+)$.
\end{lemma}

\medskip

Notice that the coefficients in system  \eqref{degenerate-eq} depends analytically on $\vv$ and $\nabla \vv$. 
We denote this system by
\[
\FF(\vv,y)=0.
\]
Now we define a new system of nonlinear degenerate elliptic  equations for vectorial function $\vv_a(z)$ as follows,
\[
\FF_a(\uu,z):=\FF(\uu(\Phi_a^{-1}(y)),y)|_{y=\Phi_a(z)}=0.
\]
The idea is that to invoke the analytic implicit function theorem for the later system, we then establish that solutions  are necessarily analytic in the parameter $a$.
It  entails that the family $\vv_a$ depends on $a$ analytically, due to the uniqueness of solutions. 
Seeing that the family of diffeomorphisms, $\Phi_a$, infinitesimally generates translations in the tangential directions, this immediately implies that the original Legendre function $\vv$ is analytic in the tangential variables.
In particular, $y'\mapsto v^1(y',0)$, which represents the free boundary, is an analytic function.

%%%%%%%%%%%%%%%%%%%%%%%%%
%%%%%SUBSECTION%%%%%%%%%%%%%

\subsection{Analyticity of the free boundary}\label{subsection:proof of main result}
Here, we prove the first part of Theorem \ref{main-result}.
The proof is based on the application of the implicit function theorem.

In order to avoid difficulties outside of $B_1^+$, we base our argument on $\ww_a:=\vv_a-\vv$. 
We first note that $\supp(\ww_a)\Subset B^+_{7/8}$, which follows from the definition of the diffeomorphism $\Phi_a$.
Moreover, $\ww_a$ solves the following degenerate elliptic system:
\[
\tilde \FF_a(\ww,z):=\FF_a(\ww+\vv,z)=0,\quad\text{in }B_1^+.
\]
Since the operator $\tilde \FF_a$ is well-defined only when $\partial_n w^1\neq -(1+\partial_n v^1)$ (recall the coefficients of $\FF$), we use another operator instead, that is defined as follows:
\[
\GG_a(\ww,z):=\eta(z)\tilde \FF_a(\ww,z)+(1-\eta(z))\cQ \ww,
\]
where  $\eta$ is a cut-off function such that $\eta\equiv1$ on $B_{7/8}^+$  and $\eta\equiv0$ in $\bR^n_+\setminus B_{1}^+$. 
In addition, the operator $\cQ=(\cQ^1, \cdots , \cQ^m)$ is
\[
\begin{split}
\cQ^1\ww := & z_n \Delta w^1+2(\kp-1)\partial_n w^1,\\
\cQ^j\ww := & z_n \Delta w^j+2\kp\partial_n w^j, \quad\text{ for }j>1.
\end{split}
\]

Now, denote  the ball in $W^{3,p}_*(\bR^n_+)$ with radius $\epsilon$ by 
\[
\cB_\epsilon:=\{w\in W^{3,p}_*(\bR^n_+):w(0)=0, \ \|w\|_{W^{3,p}_*(\bR^n_+)}<\epsilon\}.
\] 
In this context, $w(0)=0$ means that we choose a representation that vanishes in origin (recall the definition of the quotient space $W^{k,p}_*$ in Section \ref{section:degenerate}).
For any $w\in W^{3,p}_*$ we have $\nabla w\in W^{1,p}\subset C^{0,\gamma}$ when $p>n$. 
%Since $W^{3,p}_*\subset W^{2,p}\subset C^{1,\alpha}$ in bounded domain when $p>n/2$, 
So, if we choose $\epsilon$ small enough, then $|\nabla w^1(x)|<\frac12$ for every $w^1\in \cB_\epsilon$ and all $x\in B_{1}^+$, particularly $\partial_n w^1(x)+\partial_nv^1(x)>-\frac12-\e_0> -1$ (assumption \eqref{assumption-v-small}).
Moreover, for every $w^j\in \cB_\epsilon$, we have $|w^j(x)+v^j(x)|\leq \alpha$ for $j\geq 2$ and every $x\in B_{1}^+$. 
The later shows that the term $|\ww'+\vv'|_*^{q-1}$ in  the system  $\tilde \FF_a(\ww,z)=0$ is analytic when $\ww \in (\cB_\epsilon)^m$ for  sufficiently small $\epsilon$.
This shows that the operator $\GG_a$ is well-defined and analytic in $(\cB_\epsilon)^{m}$. 
We summarize the argument in the next lemma:

\begin{lemma}
If $\e_0$ in \eqref{assumption-v-small} is small enough, then  the mapping
\[
\cF:\bR^{n-1}\times(\cB_\epsilon)^{m}\rightarrow (W^{1,p}(\bR^n_+))^m,\qquad\cF(a,\uu)=\GG_a(\uu,\cdot\,)
\]
is  an analytic operator for sufficiently small $\epsilon$.
\end{lemma}

Note that 
\[
\cF(0,0)=\GG_0(0,z)=\eta(z)\tilde\FF_0(0,z)=\eta(z) \FF_0(\vv,z)=\eta(z)\FF(\vv,z)=0,
\]
and we are going to show that the existence of the family of solutions $\cF(a,\ww_a)=0$ and its analytic dependency to parameter $a$.
In sequel, we show the linearization operator of $\GG_a$ at $\ww=0$ and $a=0$, denoted by 
\[
D_{\ww} \GG_a|_{(a,\ww)=(0,0)}:(W^{3,p}_*(\bR^n_+))^m\rightarrow (W^{1,p}(\bR^n_+))^m,
\]
 is invertible. 
First, let $D_{\ww}\tilde \FF_a|_{(a,\ww)=(0,0)}=: \cQ + \cP $ be the linearization of $\tilde \FF_a$ and consider $\cP=(\cP^1, \cdots , \cP^m)$.
Since $D_z\Phi_a|_{a=0}$ is identity, then
\begin{align*}
\cP^1 \ww=&z_n \sum_\ell a_\ell \partial_{\ell n}w^1+ \sum_{i,\ell} (z_n \partial_{\ell n}v^1\partial_{v_i^1}a_\ell)\partial_i w^1+\sum_{i=1}^n \partial_{v_i^1}b \, \partial_i w^1 + \sum_{j=2}^m \partial_{v^j}b \, w^j,\\[10pt]
\cP^j \ww=&  \sum_\ell a_\ell (z_n\partial_{\ell n}w^j+\kp\partial_\ell w^j)+ \sum_{i,\ell} (z_n \partial_{\ell n}v^j+\kp \partial_\ell v^j)\partial_{v_i^1}a_\ell \partial_i w^1\\
&+ c \partial_n w^j + \sum_{i} \partial_{v_i^1}c\, \partial_n v^j \partial_i w^1 + \sum_{s} \partial_{v^s} c \,  \partial_n v^j w^s.
\end{align*}
This yields that 
\[
D_{\ww} \GG_a|_{(a,\ww)=(0,0)}=\cQ +\eta\,\cP.
\]
The operator $\cQ: (W^{3,p}_*(\bR^n_+))^m\rightarrow (W^{1,p}(\bR^n_+))^m$ is invertible by virtue of Theorem \ref{reg-deg-lap}. 
The next lemma (Lemma \ref{lem-norm-operator-eta-P}) shows that if we choose $\e_0$ in \eqref{assumption-v-small} small enough, the norm of operator $\eta\,\cP$ is small so that the operator $\cQ +\eta\,\cP$ is also invertible.

Since the function $a\mapsto \GG_a(\ww,\cdot)$ is analytic, the implicit function theorem implies that $a\mapsto \ww_a$ is analytic in $a$ in a neighborhood of $a=0$. 
In particular, $a\mapsto v^1_a(z',z_n)=v^1(\phi_{a,z}(1),z_n)$ is analytic. 
Since the derivative of $a\mapsto \phi_{a,z}(1)$ is invertible at $a=0$, so for a fix point $z$ map $a\mapsto \phi_{a,z}(1)=b$ is an analytic diffeomorphism in a neighborhood of $a=0$ in $\bR^{n-1}$ (see Remark \ref{rmk:phi-diffeomor}). Thus $b\mapsto a\mapsto v^1(b,z_n)$ must be analytic when $z_n$ is fixed and small enough.
For later reference we  summarize  the result in the following proposition.

\begin{proposition}\label{prop:tangential-analyticity}
Suppose $\vv$ is a solution of system \eqref{degenerate-eq} satisfying \eqref{assumption-v-small}. 
If $\e_0$ is small enough then $\vv$ is tangentially analytic in a neighborhood of $y=0$.  
Indeed,  map $\mathfrak{B}:B_1'\subset\bR^{n-1}\rightarrow (\cB_\epsilon)^{m}$ with $\mathfrak{B}(y')= \vv(y',\cdot)$ is analytic at $y'=0$.
\end{proposition}

To close the argument, we prove the following lemma to show the norm of $\eta\cP$ can be assumed small enough.

\begin{lemma}\label{lem-norm-operator-eta-P}
Assume that $\vv$ satisfies \eqref{v-W3p}, then the norm of operator 
$$\eta\cP:(W^{3,p}_*(\bR^n_+))^m\rightarrow (W^{1,p}(\bR^n_+))^m$$ 
for $p>n$ satisfies
\[
\norm{\eta\cP}\leq C\e_0,
\] 
where constant $C$ depends only on $m, n, p$ and $\kp$.
\end{lemma}

\begin{proof}
Based on the definition of the operator $\cP^1$ and the assumption of small values for the coefficients $a_\ell$, $b$, and $c$, we can conclude that 
%From the definition of operator $\cP^1$ and regarding to the smallness of coefficients $a_\ell$, $b$ and $c$, we get
\begin{align*}
\norm{\cP^1 \ww}_{L^p(B_1^+)}\leq 
C\e_0\left( \norm{w^1}_{W^{2,p}_*(B_1^+)} + \norm{\ww'}_{L^p(B_1^+)}  \right).
\end{align*}
Also,
\begin{align*}
\norm{\nabla (\cP^1 \ww)}_{L^p(B_1^+)}\leq &
I_1\times  \norm{w^1}_{W^{3,p}_*(B_1^+)} + I_2 \times \norm{\nabla\ww'}_{L^{p}(B_1^+)}  \\
&+ I_3 \times \norm{w^1}_{W^{2,p}(B_1^+)} + I_4 \times \norm{\nabla w^1}_{L^\infty(B_1^+)} + I_5 \times \norm{\ww'}_{L^\infty(B_1^+)} .
\end{align*}
where
\begin{align*}
I_1=&\sum_{\ell}\norm{a_{\ell}}_{L^\infty(B_1^+)}+\sum_{i,\ell}\norm{z_n \partial_{\ell n}v^1\partial_{v_i^1}a_\ell}_{L^\infty(B_1^+)}
+\sum_{i}\norm{ \partial_{v_i^1}b}_{L^\infty(B_1^+)}  \\
I_2=& \sum_{j}\norm{ \partial_{v^j}b}_{L^\infty(B_1^+)},\\
I_3=&\sum_{\ell}\norm{\nabla(z_n a_{\ell })}_{L^\infty(B_1^+)},\\
I_4=&\sum_{i,\ell}\norm{\nabla(z_n \partial_{\ell n}v^1\partial_{v_i^1}a_\ell)}_{L^p(B_1^+)}+\sum_{i}\norm{\nabla( \partial_{v_i^1}b)}_{L^p(B_1^+)}\\
I_5=& \sum_{j}\norm{ \nabla(\partial_{v^j}b)}_{L^p(B_1^+)}.
\end{align*}
According to \eqref{v-W3p}, all terms $I_1, \dots , I_5$ are comparable to $\e_0$, this along with inequalities 
$\norm{\nabla w^1}_{L^\infty(B_1^+)}\leq C\norm{w^1}_{W^{2,p}(B_1^+)}$ and $\norm{\ww'}_{L^\infty(B_1^+)}\leq C\norm{\ww'}_{W^{1,p}(B_1^+)}$ we infer that 
\begin{align*}
\norm{\eta\cP^1 \ww}_{W^{1,p}(\bR^n_+)}\leq & C\norm{\cP^1 \ww}_{W^{1,p}(B_1^+)}\leq C\e_0\left( \norm{w^1}_{W^{3,p}_*(B_1^+)} + \norm{\ww'}_{W^{1,p}(B_1^+)}  \right) \\
\leq & C\e_0\norm{\ww}_{W^{3,p}_*(B_1^+)}.
\end{align*}
We argue the same way for operators $\cP^j$, $j\ge2$.
\end{proof}

%\medskip

%%%%%%%%%%%Subsection%%%%%%%%%%%%%
\subsection{Analyticity of the solution}
In previous subsection we proved partial analyticity for the Legendre function $\vv$. Indeed, we showed that in a neighborhood of $y=0$, the function $y'\mapsto \vv(y',y_n)$ is analytic in tangential directions. 
Although this is sufficient for proving the analyticity of the regular part of free boundary, a natural question is whether a higher degree of regularity can also be obtained for $\vv$ in the vertical directions $y_n$.
This property is established by virtue of the following proposition.
For the sake of convenience, we assume from now on that the result of previous subsection, namely the analyticity of $\vv(y',y_n)$ in tangential directions,  holds in $B_1^+$.
%It also provide another proof for the previous result in Subsection \ref{subsection:proof of main result}.

\begin{proposition}\label{prop:estimate-vertical-derivative}
If $\e_0$ in \eqref{assumption-v-small} is small enough, then there exists universal constants $R_0, C_0>0$ such that for any multi-index $\mu=(\mu_1,\cdots,\mu_n)$ with $|\mu|=k$
\[
\norm{\partial^\mu \vv}_{L^\infty(B_{1}^+)}\leq C_0R_0^{-k}k!.
\]
\end{proposition}

Before proving this proposition, we explain how this can be used to show the analyticity of solution, the second part of Theorem \ref{main-result}.
If we consider the Taylor expansion for $v^j$ at $y=0$, for every $0\leq j\leq m$ the remainder term vanishes provided $|y|\leq R_0/ne$, because
\begin{align*}
\Big|v^j(y)-\sum_{k=0}^{N-1}\sum_{|\mu|=k}\frac{\partial^\mu v^j(0)}{\mu!}y^\mu\Big|
=\Big|\sum_{|\mu|=N} \frac{\partial^\mu v^j(y^*)}{\mu!}y^\mu\Big|
 \leq \frac{C_0}{N}(\frac{n|y|}{R_0})^N.
\end{align*}
Then $\vv$ is analytic in a neighborhood of $y=0$, and 
 the inverse of partial Hodograph-Legendre transform defined in Section \ref{Section:Hodograph}, $T^{-1}$, is analytic. 
 The analytic version of inverse function theorem implies that $T$ and so $(u^1)^{1/\kp}$ is an analytic function.
Thus  $u^j=u^1v^j/\alpha$ is multiplication of $u^1$ and the  analytic function  $v^j$. 
Hence,  $\uu/|\uu|$ and $ |\uu|^{\frac1{\kappa}}  $ have  the analytic  extension  from $\{|\uu|>0\}$ to $ \overline{\{|\uu|>0\} }$ near $x=0$.

\medskip

The following lemma is crucial for  proof of Proposition \ref{prop:estimate-vertical-derivative}.
\begin{lemma}\label{estimate-normal-derivative}
Suppose $u$ satisfies the ordinary differential equation
\[
x_n\partial_n^2 u + \gamma \partial_n u =f,
\]
for some $\gamma>1$ and $f\in L^\infty(0,1)$.
If $x_n\partial_n u \in L^\infty(0,1)$, then $ u\in W_*^{2,\infty}(0,1)$ and
\[\begin{split}
\norm{\partial_n u}_{L^\infty}&\leq \frac1\gamma\norm{f}_{L^\infty},\\
\norm{x_n\partial_n^2 u}_{L^\infty}&\leq 2\norm{f}_{L^\infty}.
\end{split}\]
\end{lemma}
\begin{proof}
Rewrite the equation as  $\partial_n(x_n^\gamma \partial_n u)=x_n^{\gamma-1}f$ and integrate to get
\[
\partial_n u(x) = x_n^{-\gamma}\int_0^{x_n}t^{\gamma-1}f(x',t)dt.
\]
Since $\gamma >1$ and $x_n\partial_n u \in L^\infty$, we will have $\lim_{x_n\ra0}x_n^\gamma\partial_n u=0$.
Thus
\[
\norm{\partial_n u}_{L^\infty}\leq x_n^{-\gamma}\int_0^{x_n}t^{\gamma-1}\norm{f}_{L^\infty}dt=\frac1\gamma \norm{f}_{L^\infty},
\]
and 
\[
\norm{x_n\partial_n^2 u}_{L^\infty}\leq \norm{f}_{L^\infty} + \gamma\norm{\partial_n u}_{L^\infty}\leq 2\norm{f}_{L^\infty(B_1^+)}. \qedhere
\]
\end{proof}

\begin{remark}\label{remark:ODE}
This lemma shows that operator $x_n\partial_n^2 + \gamma\partial_n: W_*^{2,\infty}(0,1) \longrightarrow L^\infty(0,1)$ is invertible. 
If we consider a perturbation operator $x_np(x)\partial_n^2 + \gamma q(x)\partial_n$ where $p, q\approx 1$, it will be also invertible.
\end{remark}

We prove Proposition \ref{prop:estimate-vertical-derivative} in a general setting when $\vv=(v^1,\dots,v^m)$ is a solution of following system
\begin{equation}\label{analytic-system}
y_n\partial_n^2 v^j + \gamma_j \partial_n v^j = y_n \sum_{k,\ell}A_{k\ell}^j(\nabla \vv) \partial_{k\ell} v^j+ g^j(\vv, \nabla \vv),\qquad j=1,\dots, m,
\end{equation}
where $\gamma_j>1$, and the  coefficients $A_{k\ell}^j$ and $g^j$ are analytic functions and satisfy
\begin{equation}\label{analytic-system-conditions}
A_{nn}^j(0)=0, %\quad g^j(0,0)=0, 
\quad  \partial_{v_n^i}g^j(0,0)=0,
\end{equation}
where $\partial_{v_n^i}$ stands for the derivative  with respect to the variable $\partial_n v^i$.

Also, we have to assume that the solution $\vv$ satisfies  \eqref{v-W3p} as well as the following estimate holds for some positive constants $R$ and $C_1$
\begin{equation}\label{prop:analytic-base-induction}
\sum_{\mu_n=0}\frac{R^{|\mu|}}{|\mu|!}\left( \norm{y_n\partial_n\partial^\mu\nabla\vv}_{L^\infty(B_1^+)} + \norm{\partial^\mu\nabla\vv}_{L^\infty(B_1^+)} \right)\leq C_1\e_0,
\end{equation}
where the summation is on all multi-indices $\mu$ with $\mu_n=0$.
As a result of the following lemma, any solution to  \eqref{degenerate-eq} will meet the condition \eqref{prop:analytic-base-induction}.

\begin{lemma}
If $\e_0$ in \eqref{assumption-v-small} is small enough, then every solution of \eqref{degenerate-eq} satisfies \eqref{prop:analytic-base-induction} for some constants $R$ and $C_1$.
\end{lemma}
\begin{proof}
We recall the result of the previous subsection, Proposition \ref{prop:tangential-analyticity},  that $b\mapsto \vv(b,y_n)$ is the analytic map $\mathfrak{B}:B_1'\subset\bR^{n-1}\mapsto W^{3,p}_*(B_1^+)\subset W_*^{2,\infty}(B_1^+)$. 
Then $\mathfrak{B}$ belongs to Gevrey class $G^1$ and there exists constants $\bar C_1$ and $R_1$ such that
\begin{equation}\label{series-B}
\sum_{\mu_n=0}\frac{R_1^{|\mu|}}{|\mu|!}\norm{\partial^\mu_b\mathfrak{B}}_{L^\infty(B_1')}\leq  \bar C_1,
\end{equation}
as well as
\begin{equation}\label{series-v0}
\sum_{\mu_n=0}\frac{R_1^{|\mu|}}{|\mu|!}\norm{\partial^\mu\vv(\cdot,0)}_{L^\infty(B_1')}\leq  \bar C_1,
\end{equation}
where the summations are on all multi-indices $\mu$ with $\mu_n=0$.
Note that for every $b$, the image of $\mathfrak{B}(b)$ is a one-dimensional function in $W^{2,\infty}_*(B_1^+)$ which depends only on $y_n$. 
Thus
\[
\norm{\partial^\mu_b\mathfrak{B}}_{L^\infty(B_1')}\cong
\norm{\partial_n\partial^\mu\vv}_{L^\infty(B_1^+)} + \norm{y_n\partial_n^2\partial^\mu\vv}_{L^\infty(B_1^+)},
\]
and so
\begin{equation*}
\sum_{\mu_n=0}\frac{R_1^{|\mu|}}{|\mu|!}\left( \norm{y_n\partial_n\partial^\mu(\partial_n\vv)}_{L^\infty(B_1^+)} + \norm{\partial^\mu(\partial_n\vv)}_{L^\infty(B_1^+)} \right)\leq \bar C_1.
\end{equation*}
On the other hand, for $i<n$
\[
\norm{\partial^\mu(\partial_i\vv)}_{L^\infty(B_1^+)} \le \norm{\partial^\mu(\partial_i\vv)(\cdot,0)}_{L^\infty(B_1')} + \norm{\partial_n\partial^\mu(\partial_i\vv)}_{L^\infty(B_1^+)},
\]
and 
\[
\norm{y_n\partial_n\partial^\mu(\partial_i\vv)}_{L^\infty(B_1^+)} \le \norm{\partial_n\partial^\mu\partial_i\vv}_{L^\infty(B_1^+)} 
\le \norm{\partial^\mu_b\partial^{i}_b\mathfrak{B}}_{L^\infty(B_1')}.
\]
All these together with \eqref{series-B} and  \eqref{series-v0} shows that 
\[
\sum_{\substack{\mu_n=0 }}\frac{R_2^{|\mu|}}{|\mu|!}\left( \norm{y_n\partial_n\partial^\mu\nabla\vv}_{L^\infty(B_1^+)} + \norm{\partial^\mu\nabla\vv}_{L^\infty(B_1^+)} \right)\leq \tilde C_1,
\]
for some $R_2<R_1$ and constant $\tilde C_1$. 
Notice that the first term in summation \eqref{prop:analytic-base-induction}, when $\mu=0$, is controlled by $C\e_0$ due to \eqref{v-W3p}.
Then  the estimate    \eqref{prop:analytic-base-induction} will be obtained by constants $C_1= C+  \tilde C_1$ and  $R= R_2\e_0$. 
\end{proof}

Before proceeding, we introduce the following notation
\[
\norm{u}_{s,R,r}:=\sum_{ |\mu|+\mu_n\leq s}\frac{R^{|\mu|}r^{\mu_n}}{|\mu|!}\left(\norm{y_n\partial_n\partial^\mu  u}_{L^\infty(B_1^+)}+(\mu_n+1)\norm{\partial^\mu u}_{L^\infty(B_1^+)}\right)
\]
where the summation is on all multi-indices $\mu$ with $ |\mu|+\mu_n=\mu_1+\cdots+\mu_{n-1}+2\mu_n\leq s$.
The following lemma is crucial for our result.

\begin{lemma}\label{lem:s-norm-estimate}
Assume that $f(s,p):\bR\times\bR^n\ra\bR$ is analytic at $(s,p)=(0,0)$ and its Taylor series is majorized by the geometric series
\[
f(s,p) \ll \frac{Mr_0}{r_0-(s+p_1+\cdots+p_n)},
\] 
for some constants $M$ and $r_0$.
Then for any $u\in W^{s+1,\infty}_*(B_1^+)$ such that $\norm{u}_{W^{1,\infty}(B_1^+)} \le r_0/2$, we have
\[
\norm{f(u,\nabla u)}_{s,R,r} \le \frac{Mr_0}{r_0-\left(\norm{u}_{s,R,r} + \sum_i\norm{\partial_i u}_{s,R,r}\right)}.
\]
% If we furthermore have $f(0,0)=0$, then
%\[
%\norm{f(u,\nabla u)}_{s,R,r} \le \frac{Mr_0\left(\norm{u}_{s,R,r} + \sum_i\norm{\partial_i u}_{s,R,r}\right)}{r_0-\left(\norm{u}_{s,R,r} + \sum_i\norm{\partial_i u}_{s,R,r}\right)}.
%\]
\end{lemma}
\begin{proof}
Consider  the Taylor expansion of $f$
\[
f(s,p)= \sum_{\alpha=(\alpha_0,\dots,\alpha_n)} F_\alpha \, s^{\alpha_0}p_1^{\alpha_1}\cdots p_n^{\alpha_n}.
\]
Hence,
\[\begin{split}
&\norm{f(u,\nabla u)}_{s,R,r} =\sum_{ |\mu|+\mu_n\leq s }\frac{R^{|\mu|}r^{\mu_n}}{|\mu|!}
\left( \norm{y_n\partial_n \partial^{\mu}f(u,\nabla u)}_\infty +(\mu_n+1) \norm{ \partial^{\mu}f(u,\nabla u)}_\infty \right)\\
&\le \sum_{\alpha}|F_\alpha|\sum_{ |\mu|+\mu_n\leq s }\frac{R^{|\mu|}r^{\mu_n}}{|\mu|!} 
\big(\norm{y_n\partial_n  \partial^{\mu} \left( u^{\alpha_0}(\partial_1 u)^{\alpha_1}\cdots(\partial_n u)^{\alpha_n}\right)}_\infty \\
& \hspace{5cm} + (\mu_n+1)\norm{  \partial^{\mu}\left( u^{\alpha_0}(\partial_1 u)^{\alpha_1}\cdots(\partial_n u)^{\alpha_n}\right)}_\infty \big) \\
& = \sum_{\alpha}|F_\alpha|\sum_{ |\mu|+\mu_n\leq s }\sum_{\mu=\sigma_1+\cdots+\sigma_{|\alpha|}} \frac{R^{|\mu|}r^{\mu_n}}{|\mu|!} \times \frac{\mu!}{\sigma_1!\cdots\sigma_{|\alpha|}!}\times\\
 &\hspace{2cm}
 \Big(\norm{ y_n\partial_n \big(\partial^{\sigma_1}u\cdots \partial^{\sigma_{\alpha_0}}u\partial^{\sigma_{\alpha_0+1}}(\partial_1 u)\cdots \partial^{\sigma_{|\alpha|}}(\partial_nu)\big)}_\infty \\
 &\hspace{2cm} + (\mu_n+1) \norm{ \partial^{\sigma_1}u\cdots \partial^{\sigma_{\alpha_0}}u\partial^{\sigma_{\alpha_0+1}}(\partial_1 u)\cdots \partial^{\sigma_{|\alpha|}}(\partial_nu)}_\infty \Big)\\
&  \le  \sum_{\alpha}|F_\alpha| \left( \sum_{ |\mu|+\mu_n\leq s } \frac{R^{|\mu|}r^{\mu_n}}{|\mu|!}( \norm{y_n\partial_n\partial^\mu u}_\infty +(\mu_n+1) \norm{\partial^\mu u}_\infty)\right)^{\alpha_0} \times \cdots \times  \\
&\hspace{2cm}  \left( \sum_{ |\mu|+\mu_n\leq s } \frac{R^{|\mu|}r^{\mu_n}}{|\mu|!} (\norm{y_n\partial_n\partial^\mu (\partial_n u)}_\infty+(\mu_n+1)\norm{\partial^\mu (\partial_n u)}_\infty)\right)^{\alpha_n}\\
&=\sum_{\alpha}|F_\alpha| \norm{u}_{s,R,r}^{\alpha_0} \cdots \norm{\partial_n u}_{s,R,r}^{\alpha_n} 
\le  \frac{Mr_0}{r_0-\left(\norm{u}_{s,R,r} + \sum_i\norm{\partial_i u}_{s,R,r}\right)}.
\end{split} \] 
\end{proof}

The main result of this subsection is established by virtue of the following proposition. 
Indeed, by the constants $R$ and $r$ that will obtained in the following  we let $R_0=rR$ to prove  Proposition \ref{prop:estimate-vertical-derivative}.

\begin{proposition}\label{prop:analyticity-general-case}
Assume that $\vv$ is a solution of \eqref{analytic-system} which satisfies \eqref{v-W3p} and \eqref{prop:analytic-base-induction}. 
Under assumption \eqref{analytic-system-conditions},  there are universal constants $0<C_0, r, R<1$ such that for all $s\in \bN$,
\begin{equation}\label{prop:estimate-derivative}
\norm{\nabla \vv}_{s,R,r} \le C_0\e_0,
\end{equation}
if $\e_0$  is small enough.
\end{proposition}

\begin{proof} We split the proof into several steps.

\medskip
\noindent
{\bf Step 1:} 
We are going to show that we can choose $r$ small enough such that  \eqref{prop:estimate-derivative} holds for constant $R$, defined in  \eqref{prop:analytic-base-induction}, and constant $C_0$, that will be determined later.  
The proof is by induction on $s$. 
The base case,  $s=1$, will be obtained by \eqref{prop:analytic-base-induction}.
Notice that for $s=1$ the multi-index can be $\partial^\mu=\partial_i$ for $1\le i<n$.

Now assume that  \eqref{prop:estimate-derivative} holds for  $s-1$. 
We will prove the desired estimate for $s$.
Consider a multi-index $\mu$ with $|\mu|+\mu_n\le s$. 
Differentiate  equation \eqref{analytic-system} with respect to $\partial^{\mu}$,
\[\begin{split}
y_n\partial_n^2 \hat v^j+(\mu_n+\gamma_j)\partial_n\hat v^j=f_\mu^1+f_\mu^2
\end{split}\]
where $\hat v^j=\partial^{\mu}  v^j$ and
\[\begin{split}
f_\mu^1:=&A^j_{nn}(y_n\partial_n^2 \hat v^j+\mu_n\partial_n\hat v^j) + \sum_{i=1}^m\left(\sum_{k,\ell} y_n\partial_{k\ell} v^j \partial_{v_n^i}A^j_{k\ell}  +
 \partial_{v^i_n}g^j\right)\partial_n\hat v^i\\
f_\mu^2:= &\sum_{k+\ell<2n}A^j_{k\ell}\partial^\mu\left( y_n\partial_{k\ell} v^j\right)
+\sum_{k,\ell}\sum\limits_{0<\tau<\mu }\permut\mu \tau \partial^{\tau}A^j_{k\ell}\partial^{\mu-\tau}(y_n\partial_{k\ell} v^j) \\
&+\sum_{k,\ell} \left(y_n\partial_{k\ell} v^j \right) \left(\partial^{\mu}A^j_{k\ell}- \sum_{i=1}^m \partial_{v_n^i}A^j_{k\ell} \partial_n\hat v^i \right)  
+\left( \partial^\mu g^j - \sum_{i=1}^m \partial_{v^i_n}g^j\partial_n\hat v^i\right) 
\end{split}\]
The induction hypothesis implies that $\hat v^j\in W_*^{1,\infty}(B_1^+)$ and we are able to apply Lemma \ref{estimate-normal-derivative}.
Although we already do not know that the right hand side is  bounded (particularly $f_\mu^1$), we consider the equation as a perturbed form and apply Lemma \ref{estimate-normal-derivative} (see Remark \ref{remark:ODE}).
Indeed, we solve the system of ODEs 
$$y_np^j(y)\partial_{n}^2\hat v^j+ \sum_{i=1}^m  q^j_i(y) \partial_{n}\hat v^i = f_\mu^2,$$ 
where
\[
p^j(y)= 1- A^j_{nn}, \qquad q^j_i(y)= (\mu_n+\gamma_j-\mu_nA^j_{nn})\delta_{ij} -  \sum_{k,\ell}y_n\partial_{k\ell} v^j \partial_{v_n^i}A^j_{k\ell}  
- \partial_{v^i_n}g^j.
\]
By \eqref{analytic-system-conditions} and \eqref{v-W3p}, we know that $p^j\approx 1$ and $q^j_i\approx (\mu_n + \gamma_j) \delta_{ij}$.
By induction hypothesis, we know that all  terms in form of $y_n\partial_n \partial^\tau\nabla  v^j$ and $\partial^\tau\nabla v^j$ are bounded when  $|\tau|+\tau_n\leq s-1$. So, $f_\mu^2$ is bounded (we also prove this fact in more details in next steps) which necessitates that  $\hat v^j\in W_*^{2,\infty}(B_1^+)$ by applying Lemma \ref{estimate-normal-derivative}. 
%On the other hand,  there is a universal constant $C$ due to \eqref{v-W3p} and \eqref{analytic-system-conditions} such that 
%\[
%\norm{f_\mu^1}_\infty \leq C\e_0 \left(\norm{y_n\partial_{n}^2\hat v^j}_\infty + \mu_n\norm{\partial_n\hat v^j}_\infty\right),
%\]
Moreover, the following estimate holds for some constant $C$,
\begin{equation*}%\label{prop:estimate-ode}
\norm{y_n\partial_n^2 \hat v^j}_\infty+(\mu_n+1)\norm{\partial_n\hat v^j}_\infty \leq C\norm{f_\mu^2}_\infty.
\end{equation*}
Summing  this inequality over $\mu$ with $|\mu| + \mu_n\leq s$  implies that
\[
\norm{\partial_n v^j}_{s,R,r}
\leq  C \sum_{\substack{|\mu|+\mu_n\leq s}}\frac{R^{|\mu|}r^{\mu_n}}{|\mu|!}\norm{f_\mu^2}_\infty.
\]

\medskip
\noindent
{\bf Step 2:}
Consider a multi-index $\mu$ with $\mu_n\ge 1$ and let $\partial^\mu=\partial_n\partial^\sigma$.
Differentiate  equation \eqref{analytic-system} with respect to $\partial^{\tau}=\partial^{\sigma}\partial_k$ for some $k=1,\dots, n-1$,
and repeat the argument in the previous step. We get
\[
\norm{y_n\partial_n\partial^\mu(\partial_k v^j)}_\infty+(\mu_n+1)\norm{\partial^\mu(\partial_k v^j)}_\infty \leq C\norm{f_{\tau}^2}_\infty.
\]
Summing over over $\mu$ with $|\mu| + \mu_n\leq s$ and $\mu_n\ge 1$ and recall \eqref{prop:analytic-base-induction} when $\mu_n=0$, 
we obtain that 
\[
\norm{\partial_k v^j}_{s,R,r}
\leq C_1\e_0 + C \sum_{\substack{|\tau|+\tau_n\leq s \\ 1\le \tau_k}}\frac{R^{|\tau|}r^{\tau_n}}{|\tau|!}\norm{f_\tau^2}_\infty.
\]
Thus
\begin{equation}\label{prop:estimate-mu_n2}
\norm{\nabla v^j}_{s,R,r}
\leq nC_1\e_0 + Cn \sum_{ |\mu|+\mu_n\leq s }\frac{R^{|\mu|}r^{\mu_n}}{|\mu|!}\norm{f_\mu^2}_\infty.
\end{equation}

\medskip
\noindent
{\bf Step 3:}
In following steps we want to estimate $\norm{f_\mu^2}_\infty$. 
We start with $A^j_{k\ell}\partial^\mu\left( y_n\partial_{k\ell} v^j\right)$ when $k+\ell<2n$ and $ |\mu|+\mu_n\leq s$.
Term $A^j_{k\ell}$ is uniformly bounded due to \eqref{v-W3p}. 
If $\mu_n\ge 1$ and $\partial^\mu=\partial^\sigma\partial_n$, we must find an estimate for  
\[
\partial^\mu\left( y_n\partial_{k\ell} v^j\right) = y_n\partial_n\partial^{\tau}(\partial_\ell v^1) + \mu_n\partial^\tau(\partial_\ell v^1)
\]
 where $\partial^\tau = \partial_{k}\partial^{\sigma}$ (we can assume that $k<n$) with the further property $|\tau|+\tau_n \leq |\sigma|+1+\sigma_n\leq s-1$. So, we are able to control them by induction assumption. 
 If $\mu_n=0$, we control these terms by \eqref{prop:analytic-base-induction}.
Thus
\[\begin{split}
\sum_{\substack{|\mu|+\mu_n\leq s }}\sum_{k+\ell<2n}&\frac{R^{|\mu|}r^{\mu_n}}{|\mu|!}\norm{\partial^\mu\left( y_n\partial_{k\ell} v^j\right)}_\infty \leq 2nC_1\e_0+ \\
\sum_{\substack{|\tau|+\tau_n\leq s-1 }}&\sum_{k+\ell<2n}\frac{R^{|\tau|}r^{\tau_n+1}}{|\tau|!}\left(\norm{y_n\partial_n\partial^{\tau}(\partial_\ell v^1) }_\infty +\mu_n \norm{\partial^\tau(\partial_\ell v^1)}_\infty \right)\\
&\leq  n^2C_1\e_0 + n^2r\norm{\nabla v^j}_{s-1,R,r} \leq n^2(C_1 + rC_0)\e_0.
\end{split}\]

\medskip
\noindent
{\bf Step 4:}
The second term in $f_\mu^2$ can be controlled as follows
\[\begin{split}
&\sum_{ |\mu|+\mu_n\leq s }\frac{R^{|\mu|}r^{\mu_n}}{|\mu|!}\sum\limits_{0<\tau<\mu }\permut\mu \tau
\norm{ \partial^{\tau}A^j_{k\ell}}_\infty \norm{\partial^{\mu-\tau}(y_n\partial_{k\ell} v^j)}_\infty \\
&\le  \left( \sum_{|\mu|+\mu_n\leq s-1 }\frac{R^{|\mu|}r^{\mu_n}}{|\mu|!}\norm{ \partial^{\mu}A^j_{k\ell}}_\infty\right)
\left(\sum_{ |\mu|+\mu_n\leq s-1 } \frac{R^{|\mu|}r^{\mu_n}}{|\mu|!}\norm{\partial^{\mu}(y_n\partial_{k\ell} v^j)}_\infty\right)\\
&=: I^1_{k\ell} \times I^2_{k\ell}.
\end{split}\]
The estimate of $I^2_{k\ell}$ when $k+\ell<2n$ has been done in Step 3. 
When $k=\ell=n$, the estimate is the same but we shall notice that  the summation here is on multi-indices $\mu$ with $|\mu|+\mu_n\le s-1$.
Thus
\[\begin{split}
I^2_{k \ell} \le &(C_1 + rC_0)\e_0, \quad\text{ when } k+\ell<2n,\\
I^2_{nn} \le  &\norm{\partial_n v^j}_{s-1,R,r}\le C_0\e_0.
\end{split}\]
In order to estimate $I^1_{k\ell}$, we apply Lemma \ref{lem:s-norm-estimate}.
We may assume that the Taylor expansion of $A^j_{k\ell}$ is majorized by the geometric series
\[
A^j_{k\ell}(\nabla \vv) \ll \frac{M R_*}{R_*-(\sum_{i,l} \partial_l v^i)},
\]
for some universal constants $M$ and $R_*$.
Furthermore,  assumption \eqref{analytic-system-conditions} for $A^j_{nn}$ yields that we may assume 
\[
A^j_{nn}(\nabla \vv) \ll 
\frac{M R_*\sum_{i,l} \partial_l v^i}{R_*-(\sum_{i,l} \partial_l v^i)}.
\]
Hence, when $k+\ell<2n$
\[\begin{split}
I^1_{k\ell}=\sum_{ |\mu|+\mu_n\leq s-1 }\frac{R^{|\mu|}r^{\mu_n}}{|\mu|!}\norm{ \partial^{\mu}A^j_{k\ell}}_\infty
 \le  \frac{M R_*}{R_*-(\sum_{i,l} \norm{\partial_l v^i}_{s-1,R,r})}\le \frac{M R_*}{R_*-mnC_0\e_0}.
\end{split}\]
%where we have chosen $\e_0$ small enough such that $\norm{\partial_l v^i}_{s-1,R,r} \le C_0\e_0\le R_*/2mn$.
Also, we have
\[
I^1_{nn} \leq \frac{M R_*\sum_{i,l} \norm{\partial_l v^i}_{s-1,R,r}}{R_*-(\sum_{i,l} \norm{\partial_l v^i}_{s-1,R,r})} 
\le \frac{MR_*mnC_0\e_0}{R_*-mnC_0\e_0} .
\]
All in all, the estimate of the second term of $f_\mu^2$ will be
\[
K(C_1+rC_0+ C_0^2\e_0)\e_0,
\]
for some universal constant $K$.

\medskip
\noindent
{\bf Step 5:}
Here we estimate the third term in $f_\mu^2$. 
Assume that $\partial^\mu=\partial^\sigma\partial_t$, then
\[\begin{split}
\partial^\mu A^j_{k\ell}=\sum_{i,l}\partial^\sigma\left( \partial_{v^i_l}A^j_{k\ell}\partial_{lt}v^i\right)
=&\sum_{i,l}\sum_{0\le \tau\le\sigma}\permut\sigma\tau\partial^\tau( \partial_{v^i_l}A^j_{k\ell})\partial^{\mu-\tau}\partial_{l}v^i.
\end{split}\] 
Hence,
\[\begin{split}
\partial^{\mu}A^j_{k\ell}- \sum_{i=1}^m \partial_{v_n^i}A^j_{k\ell} \partial_n\hat v^i=&
\sum_{i,l}\sum_{0< \tau\le\sigma}\permut\sigma\tau\partial^\tau( \partial_{v^i_l}A^j_{k\ell})\partial^{\mu-\tau}\partial_{l}v^i \\
&+\sum_i\sum_{l\ne n}  \partial_{v^i_l}A^j_{k\ell}\partial^{\mu}\partial_{l}v^i =: I_1^\mu + I_2^\mu.
\end{split}\]

Term $y_n\partial_{k\ell} v^j$ is comparable to $\e_0$ due to \eqref{prop:analytic-base-induction}. Therefore this term will be controlled by 
\[
C\e_0\sum_{ |\mu|+\mu_n\leq s }\frac{R^{|\mu|}r^{\mu_n}}{|\mu|!} \norm{I_1^\mu+I_2^\mu}_\infty.
\]
We can control term $I_1^\mu$ by
\[\begin{split}
&\sum_{\substack{|\mu|+\mu_n\leq s }}\frac{R^{|\mu|}r^{\mu_n}}{|\mu|!} \norm{I_1^\mu}_\infty \\
&\le R\sum_{i,l} \left(\sum_{\substack{|\sigma|+\sigma_n\leq s-1 }}\frac{R^{|\sigma|}r^{\sigma_n}}{|\sigma|!} \norm{\partial^\sigma \partial_{v^i_l}A^j_{k\ell}}_\infty\right) 
\left( \sum_{\substack{|\sigma|+\sigma_n\leq s-1 }}\frac{R^{|\sigma|}r^{\sigma_n}}{|\sigma|!} \norm{y_n\partial_n \partial^\sigma(\partial_l v^i)}_\infty\right) 
\\ &\leq \frac{mnMR_*RC_0\e_0}{R_*-(\sum_{i,l} \norm{\partial_l v^i}_{s-1,R,r})},
\end{split}\]
where we have assumed that 
\[
\partial_{v^i_l}A^j_{k\ell}\ll \frac{M R_*}{R_*-(\sum_{i,l} \partial_l v^i)},
\]
with the same universal constants $M$ and $R_*$.

In order to estimate $I_2^\mu$, note that $\partial_{v^i_l}A^j_{k\ell}$ is bounded and then
\[
\sum_{\substack{|\mu|+\mu_n\leq s }}\frac{R^{|\mu|}r^{\mu_n}}{|\mu|!} \norm{I_1^\mu}_\infty  
\leq C\sum_i\sum_{l\ne n} \sum_{\substack{|\mu|+\mu_n\leq s }}\frac{R^{|\mu|}r^{\mu_n}}{|\mu|!} \norm{\partial^{\mu}\partial_{l}v^i}_\infty .
\]
We distinguish two different cases $\mu_n=0$ and $\mu_n\ge 1$. 
When $\mu_n=0$, we are able to control this summation by \eqref{prop:analytic-base-induction}.
For the case $\mu_n\ge 1$ ($\partial^\mu=\partial^\sigma\partial_n$), we write  
$\partial^{\mu}\partial_{l}v^i= \partial^\tau(\partial_nv^i)$ where $\partial^\tau=\partial^\sigma\partial_{l}$. Then
\[
\sum_{\substack{|\mu|+\mu_n\leq s }}\frac{R^{|\mu|}r^{\mu_n}}{|\mu|!} \norm{I_1^\mu}_\infty  
\leq Cr\sum_i \sum_{\substack{|\tau|+\tau_n\leq s-1 }}\frac{R^{|\tau|}r^{\tau_n}}{|\tau|!} \norm{\partial^{\tau}\partial_{n}v^i}_\infty \le CnC_0r\e_0.
\]

Considering all calculations, we obtain the estimate
\[
\frac{mnMR_*RC_0\e_0^2}{R_*-mnC_0\e_0} + (CnC_0r+C_1)\e_0^2 %\le 2mnMRrC_0\e_0^2,
\]
for the third term of $f^2_\mu$.

\medskip
\noindent
{\bf Step 6:}
In this step, we are going to estimate the last term in $f_\mu^2$,
\[\begin{split}
\partial^\mu g^j - \sum_{i=1}^m \partial_{v^i_n}g^j\partial_n\hat v^i = &
\sum_{ i} \partial^\sigma(\partial_{v^i}g^j \partial_n v^i) + \sum_{\substack{i,k \\ k<n}} \partial^\sigma(\partial_{v^i_k}g^j\partial_{nk}v^i)\\
&+ \sum_i \sum_{0<\tau\le \sigma}\permut\sigma\tau \partial^\tau(\partial_{v^i_n}g^j)\partial^{\sigma-\tau}(\partial_{nn}v^i),
\end{split}\]
where we have assumed    $\mu_n\ge 1$ and  $\partial^\mu=\partial^\sigma\partial_n$. 
\[\begin{split}
&\sum_{\substack{|\mu|+\mu_n\leq s \\ 1\le \mu_n}}\frac{R^{|\mu|}r^{\mu_n}}{|\mu|!} \norm{\partial^\mu g^j - \sum_{i=1}^m \partial_{v^i_n}g^j\partial_n\hat v^i}_\infty 
\\ &\le Rr\sum_i\left( \sum_{|\sigma|+\sigma_n \le s-2} \frac{R^{|\sigma|}r^{\sigma_n}}{|\sigma|!} \norm{\partial^\sigma (\partial_{v^i}g^j )}_\infty \right)
\left( \sum_{|\sigma|+\sigma_n \le s-2} \frac{R^{|\sigma|}r^{\sigma_n}}{|\sigma|!} \norm{\partial^\sigma \partial_n v^i}_\infty\right)
\\ & +r\sum_{\substack{i,k \\ k<n}}\left( \sum_{|\sigma|+\sigma_n \le s-2} \frac{R^{|\sigma|}r^{\sigma_n}}{|\sigma|!} \norm{\partial^\sigma (\partial_{v^i_k}g^j )}_\infty \right)
\left( \sum_{|\sigma|+\sigma_n \le s-1} \frac{R^{|\sigma|}r^{\sigma_n}}{|\sigma|!} \norm{\partial^\sigma \partial_n v^i}_\infty\right)\\
&+\sum_i\left( \sum_{|\sigma|+\sigma_n \le s-2} \frac{R^{|\sigma|}r^{\sigma_n}}{|\sigma|!} \norm{\partial^\sigma (\partial_{v^i_n}g^j )}_\infty \right)\left( \sum_{|\sigma|+\sigma_n \le s-1} \frac{R^{|\sigma|}r^{\sigma_n}}{|\sigma|!} \norm{\partial^\sigma \partial_n v^i}_\infty\right)
\end{split}\]
By \eqref{analytic-system-conditions} we can assume that 
\[
\partial_{v^i_k}g^j \text{ and } \partial_{v^i}g^j \ll \frac{M R_*}{R_*-(\sum_i v^i+\sum_{i,l} \partial_l v^i)}, \quad k<n
\]
and when $k=n$,
\[
\partial_{v^i_n}g^j\ll \frac{M R_*(\sum_i v^i+\sum_{i,l} \partial_l v^i)}{R_*-(\sum_i v^i+\sum_{i,l} \partial_l v^i)}.
\]
According to Lemma \ref{lem:s-norm-estimate}, we must first estimate $\norm{v^i}_{R,r,s}$,
\[
\norm{v^i}_{R,r,s} = \sum_{|\mu|+\mu_n\leq s }\frac{R^{|\mu|}r^{\mu_n}}{|\mu|!} \norm{\partial^\mu v^i}_\infty  \le
\e_0 +\sum_{k} \sum_{\substack{|\mu|+\mu_n\leq s \\ 1\le \mu_k}} \frac{R^{|\mu|}r^{\mu_n}}{|\mu|!} \norm{\partial^\sigma(\partial_k v^i)}_\infty
\]
where $\partial^\mu=\partial^\sigma\partial_k$, then
\[
\norm{v^i}_{R,r,s} \le \e_0 +\sum_{k} R \norm{\partial_k v^i}_{R,r,s-1} \le (1+ RC_0)\e_0.
\]
Putting all these calculations together  we can deduce  the following estimate
\[\begin{split}
\sum_{\substack{|\mu|+\mu_n\leq s \\ 1\le \mu_n}}&\frac{R^{|\mu|}r^{\mu_n}}{|\mu|!} \norm{\partial^\mu g^j - \sum_{i=1}^m \partial_{v^i_n}g^j\partial_n\hat v^i}_\infty \le  
K(r+\e_0)C_0\e_0,
\end{split}\]
for some universal constant $K$.

Next, repeat the calculation when $\mu_n=0$ and recall  \eqref{prop:analytic-base-induction} to get
\[\begin{split}
\sum_{\substack{|\mu|+\mu_n\leq s \\  \mu_n=0}}&\frac{R^{|\mu|}r^{\mu_n}}{|\mu|!} \norm{\partial^\mu g^j - \sum_{i=1}^m \partial_{v^i_n}g^j\partial_n\hat v^i}_\infty \le
 \tilde K(r+\e_0)C_1\e_0.
\end{split}\]

\medskip
\noindent
{\bf Step 7:}
Put all the estimate of $f_\sigma^2$ in \eqref{prop:estimate-mu_n2}  to get 
\[\begin{split}
\norm{\nabla \vv}_{s,R,r}\leq (K_1+K_2C_0(r+\e_0+C_0\e_0) ) \e_0
\end{split}\]
where constants $K_1$ and $K_2$ are universal and depends only on $m, n, M$ and $C_1$. 
Suppose now  that $C_0= 2K_1$ and select both
 $r$ and $\e_0$ small enough such that 
 $$K_2(r+\e_0) +2K_1K_2\e_0 \le 1/2,$$
to arrive at 
 $\norm{\nabla \vv}_{s,R,r}\leq C_0\e_0$.
\end{proof}

%%%%%%%%%%%Subsection%%%%%%%%%%%%%
\subsection{Existence of  solutions} 
In this part we will prove Theorem \ref{thm:existence}. 
Indeed, it is a straightforward conclusion of a  Cauchy-Kowalevski type statement in the following proposition. 

\begin{proposition}
Consider the nonlinear degenerate elliptic system \eqref{analytic-system} with property \eqref{analytic-system-conditions}.
There exists $\e_0$ such that if 
the analytic Cauchy data $\vv_0:\bR^{n-1}\ra\bR^m$ satisfies
\begin{equation}\label{cauchy-data-assumption}
|\vv_0(0)|\le \e_0, \quad
|\nabla' \vv_0(0)|\le \e_0,
\end{equation}
then \eqref{analytic-system} has a unique  analytic local solution $\vv$ with $\vv(y',0)=\vv_0(y')$.
\end{proposition}

\begin{proof}
First, we show that we can obtain $\partial_n\vv$ uniquely at a neighborhood of  $y=0$. 
Consider  \eqref{analytic-system}  at $y_n=0$, we get
\[
\gamma_j\partial_n v^j =  g^j(\vv_0, \nabla' \vv_0,\partial_n\vv|_{y_n=0}).
\]
Apply the implicit function theorem with regard to condition \eqref{analytic-system-conditions} to determine the analytic funstion $\nabla\vv$ uniquely. 
Furthermore, it will satisfy
\[
|\nabla\vv(0)|\le c\e_0.
\]

Now we compute the Taylor expansion of $\nabla \vv$ at $y=0$. Substitute 
\[
\partial_k v^j(y)=\sum_{\mu}\frac{V^{jk}_\mu}{\mu!} \, y^\mu,
\]
in equation \eqref{analytic-system} and calculate the coefficients $V^{jk}_\mu$. 
For $\mu_n=0$, the coefficients will be obtained uniquely  from  the Cauchy data $\vv_0$ and $\partial_n\vv$. 
%When  $\partial^\mu=\partial_n$, we look at \eqref{analytic-system}  at $y=0$ to find out
%\[
%\gamma_jV^j_\mu =  g^j(\vv, \nabla \vv)\Big|_{y=0}.
%\]
%This is a system with unknown parameters $V^1_\mu,\dots, V^m_\mu$. 
%Apply the implicit function theorem with regard to condition \eqref{analytic-system-conditions} to obtain $V^j_\mu$ uniquely and furthermore they are comparable to $\e_0$.

Suppose $V^{jk}_\mu$ is known  when $|\mu|+\mu_n \le s-1$ and $j=1,\dots m$, we want to calculate $V^{jk}_\mu$ for some multi-index $\mu$ with $|\mu|+\mu_n=s\ge 2$. 
If we differentiate \eqref{analytic-system} with respect to $\partial^{\mu}$ at $y=0$, we will have 
\begin{equation}\label{Taylor:system-Vj}
\begin{split}
(\mu_n+\gamma_j)V^{jn}_\mu = &
\mu_n \sum_{k,\ell}\partial^{\sigma}\left(A_{k\ell}^j(\nabla \vv) \partial_{k\ell} v^j\right)\Big|_{y=0}+ \partial^{\mu} g^j(\vv, \nabla \vv)\Big|_{y=0}\\
= &\mu_n A_{nn}^j(\nabla\vv(0))V^{jn}_\mu + \sum_{i=1}^m \partial_{v^i_n}g^j(\vv(0),\nabla\vv(0)) V^{in}_\mu + \text{ l.o.t}
%&+ \text{ some terms consists of }V^i_t \text{ for }. 
\end{split}
\end{equation}
where l.o.t consists of $V^{ik}_{\tau} $ for  $ t\le s-1$, $1\le k\le n$, $1\le i\le m$ and multi-index $\tau$ whit $|\tau|+\tau_n\le s-1$. 
(We will see this fact later.)
In addition, we have used the notation $\partial^\mu=\partial^\sigma\partial_n$ when $\mu_n\ge 1$.

The property \eqref{analytic-system-conditions} along with condition \eqref{cauchy-data-assumption} (and also the first part of proof that $\nabla\vv(0)$ is small) implies that $A_{nn}^j(\nabla\vv(0))$ and $\partial_{v^i_n}g^j(\vv(0),\nabla\vv(0))$ can be small enough. Hence we can solve  system \eqref{Taylor:system-Vj} (for $j=1,\dots,m$) to find $V^{jn}_\mu$ uniquely. 
Now differentiate \eqref{analytic-system} with respect to $\partial^{\sigma}\partial_k$ (for $k<n$) at $y=0$ and repeat the above calculations, to  find out coefficients $V^{jk}_\mu$ uniquely.
Moreover, we conclude that 
\begin{equation}\label{Taylor:estimate-Vj}
|V^{jk}_\mu| \le  C\,|\text{l.o.t}|,\quad\text{ for }1\le j\le m\text{ and } 1\le k\le n,
\end{equation}
for some universal constant $C$.

So far, we have obtained the Taylor expansion of the solution uniquely. We must show the convergence of the series near $y=0$.
‌‌To do this, we follow the similar idea to what we did in the previous subsection. 
By virtue of analyticity of Cauchy data $\vv_0$ as well as $\partial_n\vv$, we can choose $R_1$ such that 
\[
\sum_{\mu_n=0}\frac{|V^{jk}_\mu|}{\mu!} R_1^{|\mu|} \le \tilde C_1<\infty, \quad\text{ for }1\le j\le m\text{ and } 1\le k\le n.
\]
Recall \eqref{cauchy-data-assumption} and set $R=R_1\e_0$, we get
\begin{equation}\label{taylor-expansion-estimate0}
\sum_{\mu_n=0}\frac{|V^{jk}_\mu|}{\mu!} R^{|\mu|} \le C_1\e_0, \quad\text{ for }1\le j\le m\text{ and } 1\le k\le n.
\end{equation}
We claim that there exists $0<r<1$ and $C_0$ such that
\begin{equation}\label{Taylor-convergence}
\norm{\partial_kv^j}_{s,R,r}:=\sum_{|\mu|+\mu_n\le s} \frac{|V^{jk}_\mu|}{|\mu|!\mu_n!} R^{|\mu|} r^{\mu_n}\leq C_0\e_0,
\end{equation}
for every $ s$. 
The idea is to estimate inductively l.o.t in \eqref{Taylor:estimate-Vj}.
Indeed, for $s\ge 2$
\begin{equation*}%\label{Taylor:estimate-Vj}
\begin{split}
\text{l.o.t }=  &\sum_{k+\ell<2n}\sum_{\tau\le \sigma}\permut {\sigma} \tau \mu_n\partial^\tau A_{k\ell}^j \partial^{\sigma-\tau}\partial_{k\ell} v^j
+\sum_{0<\tau\le \sigma}\permut {\sigma} \tau \mu_n \partial^\tau A_{nn}^j \partial^{\mu-\tau}\partial_n v^j\\
&+ (\partial^{\mu} g^j-\sum_{i=1}^m \partial_{v^i_n}g^j V^i_\mu)=:I_1^\mu + I_2^\mu + I_3^\mu,
\end{split}
\end{equation*}
where all functions are restricted to $y=0$.
We start with the first term, $I_1^\mu$, and similar to Step 4 in Proposition \ref{prop:analyticity-general-case} we get
\[\begin{split}
&\sum_{\substack{|\mu|+\mu_n\le s\\ 1\le\mu_n}} \frac{ R^{|\mu|} r^{\mu_n}}{|\mu|!\mu_n!}|I_1^\mu| \le \\
&r \sum_{k+\ell<2n}\left(\sum_{|\mu|+\mu_n\le s-2} \frac{R^{|\mu|} r^{\mu_n}}{|\mu|!\mu_n!}\big|\partial^\mu A_{k\ell}^j\big|_{y=0}\right)
\left(\sum_{|\mu|+\mu_n\le s-1} \frac{R^{|\mu|} r^{\mu_n}}{|\mu|!\mu_n!}|V^{jk}_\mu| \right) \\
&\le \sum_{k+\ell<2n}\frac{MR_*r\norm{\partial_kv^j}_{s,R,r}}{R_*-\sum_{i,l}\norm{\partial_lv^i}_{s,R,r}} \le K_1rC_0\e_0,
\end{split}\]
where $K_1$ is a universal constant. 

The second term, $I_2^\mu$ will be controlled similarly to  Step 5 in Proposition \ref{prop:analyticity-general-case}  as follows,
\[\begin{split}
\sum_{\substack{|\mu|+\mu_n\le s\\ 1\le\mu_n}}& \frac{ R^{|\mu|} r^{\mu_n}}{|\mu|!\mu_n!}|I_2^\mu| \le 
\left(\sum_{|\mu|+\mu_n\le s-2} \frac{R^{|\mu|} r^{\mu_n}}{|\mu|!\mu_n!}\big|\partial^\mu A_{nn}^j\big|_{y=0}\right)
\left(\sum_{|\mu|+\mu_n\le s-1} \frac{R^{|\mu|} r^{\mu_n}}{|\mu|!\mu_n!}|V^{jn}_\mu| \right)\\
\leq & \frac{MR_*\sum_{i,l}\norm{\partial_lv^i}_{s,R,r}}{R_*-\sum_{i,l}\norm{\partial_lv^i}_{s,R,r}}\norm{\partial_kv^j}_{s,R,r} \le K_2 (C_0\e_0)^2.
\end{split}\]

Finally,  using the same approach as Step 6 in Proposition \ref{prop:analyticity-general-case} the third term can be estimated,
\[\begin{split}
\sum_{\substack{|\mu|+\mu_n\le s\\ 1\le\mu_n}}& \frac{ R^{|\mu|} r^{\mu_n}}{|\mu|!\mu_n!}|I_3^\mu| \le 
\frac{CMR_*(r+\sum_i \norm{v^i}_{s,R,r}+\sum_{i,l}\norm{\partial_lv^i}_{s-1,R,r}) }{R_*-(\sum_i \norm{v^i}_{s,R,r}+\sum_{i,l}\norm{\partial_lv^i}_{s-1,R,r})}\norm{\partial_n v^i}_{s-1,R,r}\\
&\le K_3C_0\e_0(r+\e_0).
\end{split}\]

All these estimate yields that
\[
\norm{\partial_kv^j}_{s,R,r} \le C_1\e_0 + K_1rC_0\e_0+ K_2(C_0\e_0)^2 + K_3C_0\e_0(r+\e_0) \le C_0\e_0,
\]
when $C_0=2C_1$ and $r, \e_0$ has been chosen sufficiently small.
\end{proof}

\begin{proof}[Proof of Theorem \ref{thm:existence}]
Assume $0\in \Gamma$ and let $\nu$ be the orthogonal unit vector on $\Gamma$. 
By a rotation in $\bR^n$ we may assume that $\nu(0)=\ee_n$ and the hypersurface $\Gamma$ is the graph of analytic function $\tilde v^1_0:\bR^{n-1}\ra\bR$ in a neighborhood of origin. 
We also  change the coordinate in the target space $\bR^m$ if necessary, and suppose that $V_+(0)=\ee_1$.
In order to solve \eqref{degenerate-eq}, define the Cauchy data $\vv_0=(v_0^1,\dots,v_0^m)$ that $v_0^1=\tilde v_0^1-y_n$ and
$v_0^j=(V_+\cdot\ee_j)/(V_+\cdot\ee_1)$ for $j=2,\dots,m$. 
We have $\vv_0(0)=0$ and $\nabla' v^1_0(0)=0$.
Notice that \eqref{degenerate-eq} is invariant under the following scaling 
\[
v^1_r(x)=\frac{v^1(rx)}r,\quad v^j_r(x)=v^j(rx),
\]
so  $|\nabla' v^j_r|$ is small enough for appropriate values of $r$. 
Hence, condition \eqref{cauchy-data-assumption} holds for $\vv_{0,r}$ and  
 \eqref{degenerate-eq} with Cauchy data $\vv_{0,r}$ has a local solution. Then rescale the solution back to the original problem with Cauchy data $\vv_0$.
Finally,  by the inverse of Hodograph-Legendre transform defined in Section \ref{Section:Hodograph} we will find the analytic local solution $\uu$.
\end{proof}

%%%%%%%%%%%%%%%%%%%%%%%
%%%%%%%%SECTION%%%%%%%%%%
%%%%%%%%%%%%%%%%%%%%%%%

\section{Appendix}\label{Appendix}
We give a short proof for Theorem \ref{reg-deg-lap} here.  We first  introduce some auxiliary notion. 

We equip the upper half space $\bR^n_+:= \{ x\in \bR^n, x_n >0 \} $ with the Riemannian metric 
\[  g_x(v,w) = x_n^{-1} v \cdot w.\]
The length of a path $ \gamma\in C^1([a,b], \bR^n_+)$ in the upper half plane is 
\[ L(\gamma) = \int_a^b  |\gamma'| \gamma_n^{-1/2} dt   \]
and the corresponding distance 
\[  d(x,y)= \inf \{ L(\gamma): \gamma(a) = x, \gamma(b)=y\}.  \]
On vertical lines we have 
\[  d( (0,t), (0,s)) =  2 |t^{1/2}-s^{1/2}|, \] 
hence, if $ 0<x_n\le y_n \le T $ 
\[   d(x,y) \le   4 T^{1/2} -2 x_n^{1/2}-2y_n^{1/2}+  |x'-y'| T^{-1/2}. \]
Similarly, if $\gamma$ connects $x$ and $y$ and $T$ is the maximal height of $\gamma$ then 
\[  d(x,y) \ge   \max\{ 4 T^{1/2}- 2x_n^{1/2} -2 y_n^{1/2} , |x-y|T^{-1/2} \}.  \]
Optimization with respect to $T$ gives 
\begin{equation} d(x,y) \sim \frac{|x-y|}{\sqrt{ |x-y|+ x_n +y_n}} . \end{equation}

We also denote intrinsic balls by $B_r(x)$. Let 
\[  |A|_\gamma =   \int_A x_n^{\gamma} dx.  \]
Then 
\[ |B_r(x)|_{\gamma} \sim   (x_n+r^2)^{\gamma+\frac{n}2}   r^n.  \]

We claim that there exists a Green's function $G$ which satisfies (with a minor modification when $ n=\gamma=1$)  
\begin{equation} |B_{d(x,y)}(x) |_{\gamma-1}(x_n+ y_n+ d(x,y)^{2})^{(|\alpha|+|\beta|)/2}  d(x,y)^{|\alpha|+|\beta|-2}  \Big|         \partial_y^\alpha  \Big(  y_n^{\gamma-1} \partial_x^\beta G(x,y) \Big)    \Big| \le c_{|\alpha|+|\beta|}. \end{equation}
We begin with $n=1$, where, if $ \gamma \ne 1$ 
\[ G(x,y) =   \left\{   \begin{array}{rl} \frac1{1-\gamma} x^{1-\gamma}  & \text{ if } x>y \\ 
\frac1{1-\gamma} y^{1-\gamma} & \text{ if } x<y 
\end{array}  \right.      \]
resp. $\log(x)$ and $\log(y)$ if $ \gamma=1$. 
For $n \ge 2$ we rely on two following estimates, a standard elliptic estimate, and a boundary regularity estimate. 

\begin{lemma} \label{lem:local} 
Consider  $A \subset \mathbb{R}^n_+$ open, $ u \in H^{1}_{loc}(A)$
with $ x_n^{\gamma/2} \nabla u \in L^2(A)$. Suppose 
\[ \rdiv( x_n^{\gamma} \nabla u) = 0 \qquad \text{ on } A. \]
If $ r \le \frac12$ and $ A= B_{r}(e_n)$ then 
\[ \Vert \partial^\alpha u \Vert_{L^\infty(B_{r/2}(e_n))} \le  c_{|\alpha|} r^{1-|\alpha|-n/2} 
\Vert \nabla u \Vert_{L^2(B_r(e_n)) }.
\]
Moreover, if $A= B_1^+(0)$ and if $ \alpha$ is a multiindex of length at least $1$ then 
\[ \Vert \partial^\alpha u \Vert_{L^\infty (B_{1/2}^+(0))} \le c_{|\alpha|} \Vert \nabla u \Vert_{L^2(B_1^+(0), x_n^\gamma)}. \] 
\end{lemma} 
In this setting Euclidean balls and intrinsic balls are equivalent. 
\begin{proof} 
The first estimate is a standard elliptic estimate.  The weighted Poincar\'e inequality 
\[ \inf_\lambda \Vert u-c \Vert_{L^2(x_n^\gamma)(B_1^+(0),x_n^{\gamma})}  \le c 
\Vert \nabla u \Vert_{L^2(B_1^+(0), x_n^\gamma)} \]
and the Sobolev inequality 
\[ \Vert  u \Vert_{L^\infty(B^+_{1/2}(0))} 
\le c \Big( \sum_{|\alpha|\le N} \Vert \partial^\alpha u \Vert^2_{L^2(B^+_{1/2}(0), x_n^\gamma) }\Big)^{1/2} 
\]
for $N$ sufficiently large reduce 
 the second estimate to the $L^2$ estimate 
\[ \Vert \partial^{\alpha} u \Vert_{L^2(B^+_{1/2}(0), x_n^\gamma)}
\le c \Vert u \Vert_{L^2(B_1^+(0),x_n^\gamma)}.
\]
The energy estimate 
\[   \Vert \nabla u \Vert_{L^2(B^+_{1/2}, x_n^\gamma)} \le c \Vert u \Vert_{L^2(B^+_{1}(0), x_n^\gamma)} \]
can be proven by a variation of standard arguments. We apply it to tangential finite differences and tangential derivatives iteratively to deduce the desired estimate for tangential derivatives. Now we rewrite the equation as 
\[ \partial_x x_n^{\gamma} \partial_n u = - \sum_{j=1}^{n-1}  x_n^\gamma \partial^2_j u. \]
We complete the proof by arguing as in Lemma \ref{estimate-normal-derivative}.
\end{proof}

We consider weak solutions to
\[   \rdiv( x_n^{\gamma} \nabla u )= \rdiv ( x_n^\gamma F ).   \]
Let $H= \dot H^1_\gamma (x_n^\gamma)$ the Hilbert spaced defined by  the inner product 
\[ \langle u,v \rangle = \int x_n^\gamma \nabla u \cdot \nabla v dx.  \]
We consider equivalence classes modulo constants. The map 
\[  \dot H^1_\gamma \ni \phi \to   \int x_n^\gamma \nabla \phi F dx \]
is an element of the dual space. By the Riesz representation theorem there exits a unique $u\in \dot H^1_\gamma$ so that 
\[ \langle u, \phi \rangle = \int x_n^\gamma \nabla u \nabla \phi dx = \int x_n^\gamma F \nabla \phi dx \]
for all $ \phi \in \dot H^1_\gamma$. The map 
\[ (L^2(x_n^\gamma))^n  \ni F \to u \in \dot H^1_\gamma \] 
is linear, continuous. The composition 
\[  A : \dot H^1_\gamma \ni v \to F=\nabla v \to u \in \dot H^1_\gamma \]
is self adjoint. 
Let $ x,y $ be in the upper half plane and $ F$  supported in $B_{d(x,y)/4} (y)$ and let $u$ be the solution. Then, by rescaling Lemma \ref{lem:local}  
\[ |B_{d(x,y)/4}(x)|_{\gamma}^{1/2}  d(x,y)^{|\alpha|-1}  \Vert (x_n^{1/2} + d(x,y) )^{|\alpha|} \partial^\alpha u \Vert_{L^\infty(B_{d(x,y)/4})}  
\le  c_{|\alpha|} \Vert x_n^{\gamma/2} F \Vert_{L^2} 
\]
for $ \alpha \ne 0$. 
We define $\dot H^1_\gamma(A)$ in the obvious way and 
\[ H^1_\gamma = \dot H^1_\gamma \cap L^2(x_n^{\gamma-1}). \]
 There is variant of the Poincar\'e inequality, 
\begin{lemma} 
Let $B=B_r(x)$ be an intrinsic ball. Then 
\[    \Vert u-u_{B}   \Vert_{L^2(B, x_n^{\gamma-1}) } \le  c  (1+ x_n)^{-1/2}   r^{-1}  
\Vert u \Vert_{\dot H^1_\gamma(B) }. \]
\end{lemma}
Given disjoint balls $B_1,B_2$ we define the map 
\[A_{B_1,B_2 }: H^1_\gamma (B_1) \ni v \to  F = \nabla v \to u- u_{B_2} \in H^1_\gamma(B_2) \]
so that 
\[  A_{B_1,B_2}^* = A_{B_2,B_1} \]
Let $ x\in B_2$. The map 
\[  v \to   \partial^\alpha Av (x) \]
lies in $(H^1_\gamma(B_1))^*$. We apply the Riesz representation theorem to obtain 
\[  \partial^\alpha Av(x) = \int_{B_1} \partial_x^\alpha G(x,y)\rdiv(  y_n^\gamma \nabla v) dx     \] 
and by selfadjointness 
\[  y \to w(y) = \partial_x^\alpha G(x,y) \]
satisfies 
\[  \rdiv( x_n^\gamma \nabla w )= 0 \]
in $B_1$. We obtain for $ |\alpha| \ge 1 $ and $ |\beta| \ge 1$
\[  \Big(  ( x_n^{\frac12} +y_n^{\frac12} + d(x,y))   d(x,y)\Big)^{|\alpha|+|\beta|-2} |B_{d(x,y)}(x)|_\gamma\,  | \partial^{\alpha}_y  \partial^{\beta}_x  G(x,y) | \le c       \]
In particular, if $ |\alpha|\ge 1$. 
\[  \Big(  ( x_n^{\frac12} +y_n^{\frac12} + d(x,y))   d(x,y)\Big)^{|\alpha|-1} |B_{d(x,y)}(x)|_\gamma\,  | \partial^{\alpha}_y  \partial_{x_j}    G(x,y) | \le c.        \]
This allows to fix the constant by assuming 
\[ \lim_{x_n \to \infty}  \partial_{y_j} G(x,y)=0  \]
hence, if $ |\alpha|\ge 1$ 
\[   |B_{d(x,y)}(x)|_\gamma \Big( (x_n^{\frac12} +y^{\frac12}_n + d(x,y))(d(x,y)) \Big)^{|\alpha|}   |  \partial^\alpha_y G(x,y) | \le c.   \]
We recall that $ n \ge 2$ and repeat the argument and normalize 
\[ \lim_{x_m\to \infty} G(x,y) = 0. \]
If $ n \ge 3$ we arrive at 
\begin{equation}   |B_{d(x,y)}(x)|_{\gamma-1}  d^2(x,y)  |  G(x,y) | \le c.   \end{equation}
and if $n=2$
\begin{equation}   |B_{d(x,y)}(x)|_{\gamma-1}  d^2(x,y)     ) |    |  G(x,y) | \le c \log\Big( 2+ \frac{x_n+y_n}{d^2(x,y)}\Big) .  \end{equation}

\begin{proof}[Proof of Theorem \ref{reg-deg-lap}]
First, we claim that (have to make the conditions on $u$ precise) 
\begin{equation} \label{eq:L2estimate}    \Vert x_n D^2 u \Vert_{L^2(x_n^{\gamma-1})}  + \Vert Du \Vert_{L^2(x_n^{\gamma-1})} \le c \Vert x_n \Delta u + \gamma \partial_n u \Vert_{L^2(x_n^{\gamma-1})} \end{equation}
We do a formal argument, which can be easily justified. First 
\[  \int \rdiv ( x_n^\gamma \nabla u) \partial_n u dx 
= - \int x_n^\gamma \nabla u \nabla\partial_{n} u dx 
= \frac{\gamma}2 \int x_n^{\gamma-1} |\nabla u|^2 dx 
\]
and by Cauchy-Schwarz 
\[ \Vert \nabla u \Vert_{L^2(x_n^{\gamma-1})} \le \frac{2}{\gamma} \Vert x_n \Delta u + \gamma u_n \Vert_{L^2(x_n^{\gamma-1})}\] 
Similarly we multiply with $x_n \Delta u$ and integrate by parts to arrive at \eqref{eq:L2estimate}.

Consider 
\[    \rdiv( x_n^\gamma \nabla u ) = x_n^{\gamma-1} f \] 
and let $T$ by any of the maps 
\[    L^2(x_n^{\gamma/2} ) \ni f \to   \partial_j u, x_n \partial^2_{jk} u \in L^ 2(x_n^{\gamma/2}). \]
It satisfies 
\[ \Vert Tf \Vert_{L^2(x_n^{\gamma-1})} \le c \Vert f \Vert_{L^2(x_n^{\gamma-1})}. \]
It has an integral kernel $k$ which satisfies 
\[   | B_{d(x,y)}(x)|_{\gamma-1} \le C      \]
\[   |B_{d(x,y)}(x)|_{\gamma-1} ( |k(x,y)-k(\tilde x, \tilde y)| ) \le c \frac{d(x,\tilde x), d(y,\tilde y)}{d(x,y) + d(\tilde x, \tilde y) } .  \]
The upper half plane equipped with the metric $d$ and the measure $ x_n^{\gamma-1} dx $ is doubling: 
\[  |B_{3r}(x)|_{\gamma-1} \le c |B_r(x)|_{\gamma-1}. \]
Thus $T$ and its adjoint are Calder\'on-Zygmund operators (see \cite{stein1993harmonic}) and satisfy 
\[  \Vert Tf \Vert_{L^p(x_n^{\gamma-1}) } \le c \Vert f \Vert_{L^p(x_n^{\gamma-1}) }. \]
Let us recall the notion of Muckenhoupt weights: Let $1 < p < \infty$. We say that $ \omega \in A_p $ if and only if 
\[ \sup_B  \left( |B|^{-1}_{\gamma-1}  \int_B  \omega dx \right) \left( |B|^{-1}_{\gamma-1} \int_B \omega^{-\frac{p'}{p}}\right)^{\frac{p}{p'}} < \infty.   \]
See \cite{stein1993harmonic} for a Muckenhoupt weights in $R^n$ and \cite{MR2990061} for an extensions to spaces of homogeneous type which cover our setting. 
If $ 1< p < \infty $ and $ \omega \in A_p $ then 
\[ \int |Tf|^p \omega x_n^{\gamma-1} dx \le c \int |f|^p \omega x_n^{\gamma-1} dx.  \]
Is is not hard to  verify  that $ x_n^{1-\gamma}$ is a Muckenhoupt weight. 

This completes the proof of the first part. 
For the second part, differentiate the equation with respect to a tangential direction $\partial_k$, $k<n$. Then $\partial_k u$ satisfies a similar equation and according to the first part,
\[
\norm{\partial_ku}_{W^{2,p}_*(\bR^n_+)}\leq C_{n, p, \gamma}\norm{\partial_kf}_{L^p(\bR^n_+)}.
\]
Now if we differentiate \eqref{deg-lap} again with respect to $\partial_n$, we have 
\[
 y_n \Delta \partial_n u + (\gamma + 1)\partial_n(\partial_n u) = \partial_n f - \Delta' u.
\]
Therefore,
\[
\norm{\partial_nu}_{W^{2,p}_*(\bR^n_+)}\leq C_{n, p, \gamma+1}\norm{\partial_nf}_{L^p(\bR^n_+)}+\sum_{k=1}^{n-1}\norm{\partial_ku}_{W^{2,p}_*(\bR^n_+)},
\]
and the estimate holds for $\tilde C_{n, p, \gamma}=2C_{n, p, \gamma}+C_{n, p, \gamma+1}$.
\end{proof}

\medskip

\paragraph{\bf{Acknowledgements.} }
Morteza Fotouhi was supported by Iran National Science Foundation (INSF) under project No. 99031733. Herbert Koch was partially supported by
the Deutsche Forschungsgemeinschaft (DFG, German Research Foundation) through the
Hausdorff Center for Mathematics under Germany’s Excellence Strategy - EXC-2047/1 -
390685813 and through CRC 1060 - project number 211504053.

\section*{Declarations}

\noindent {\bf  Data availability statement:} All data needed are contained in the manuscript.

\medskip
\noindent {\bf  Funding and/or Conflicts of interests/Competing interests:} The authors declare that there are no financial, competing or conflict of interests.

\end{document}